\documentclass[
reqno]{amsart}
\usepackage{amscd,amsmath,amssymb,amsthm,amsfonts}
\usepackage{lipsum}
\usepackage{xpatch}
\xpatchcmd{\paragraph}{\normalfont}{{\normalfont\bfseries}}{}{}
\usepackage{cancel}
\usepackage{color,xcolor}
\usepackage{graphics}
\usepackage{inputenc}
\usepackage{fontenc}

\usepackage[english]{babel}

\newcommand{\myspace}{\qquad\qquad\qquad}

\newcommand{\cA}{{\mathcal A}}

\newcommand{\cD}{{\mathcal D}}
\newcommand{\cF}{{\mathcal F}} 

\newcommand{\cH}{{\mathcal H}}
\newcommand{\cK}{{\mathcal K}}
\newcommand{\cL}{{\mathcal L}}

\newtheorem{theorem}{Theorem}[section]

\newtheorem{proposition}[theorem]{Proposition}
\newtheorem{remark}[theorem]{Remark}
\newtheorem{remarks}[theorem]{Remarks}

\numberwithin{equation}{section}

\usepackage{mathrsfs}
\usepackage[margin=37mm]{geometry}
\newcommand{\R}{\mathbb{R}}

\date{}

\begin{document}

\title[The Cauchy-Dirichlet problem for the MGT equation]{The Cauchy-Dirichlet problem for the Moore-Gibson-Thompson equation}

\author{Francesca Bucci}
\address{Francesca Bucci, Universit\`a degli Studi di Firenze,
Dipartimento di Matematica e Informatica,
Via S.~Marta 3, 50139 Firenze, ITALY
}
\email{francesca.bucci@unifi.it}

\author{Matthias Eller}
\address{Matthias Eller, Georgetown University, 
Department of Mathematics and Statistics,
Georgetown~360, 37th and O Streets NW, Washington DC 20057, USA
}
\email{Matthias.Eller@georgetown.edu}



\begin{abstract}
The Cauchy-Dirichlet problem for the Moore-Gibson-Thompson equation is analyzed. 
With the focus on non-homogeneous boundary data, two approaches are offered: one is based on the theory of hyperbolic equations, while the other one uses the theory of operator semigroups. 
This is a mixed hyperbolic problem with a characteristic spatial boundary. 
Hence, the regularity results exhibit some deficiencies when compared with the non-characteristic case.
%
\end{abstract}

\maketitle


\section{Introduction}
The Moore-Gibson-Thompson (MGT) equation
\begin{equation}\label{e:mgt}
w_{ttt} + \alpha w_{tt} - c^2\Delta w -b \Delta w_t =f
\end{equation}
is a hyperbolic partial differential equation with respect to the $t$ variable, of order three for all $b>0$.
This can be seen as follows: its principal symbol is 
$-{\rm i}\xi_0^3 + b{\rm i} \xi_0 |\xi|^2$, and hence, the equation is even strictly hyperbolic, since its characteristic roots $0,\sqrt{b}|\xi|,-\sqrt{b}|\xi|$ are distinct for all $\xi \in \R^d \setminus \{0\}$. The other terms are of lower order.
Since the principal part is strictly hyperbolic, hyperbolicity is preserved regardless of
lower order terms; see \cite[Corollary~12.4.10]{hormander83}. 
In this article we discuss the Cauchy-Dirichlet problem for the equation \eqref{e:mgt}, that means the initial-boundary value problem (IBVP)
\begin{equation*}
Mw =f \;\; \text{in $Q$,} \quad w=g \;\; \text{in $\Sigma$,} \quad 
w(0)=w_0\,, \; w_t(0)=w_1\,, \; w_{tt}(0)=w_2 \;  \text{in $\Omega$.}
\end{equation*}
having set 
\begin{equation}\label{e:differential-operator}
Mw := w_{ttt} + \alpha w_{tt} - c^2\Delta w -b \Delta w_t\,,
\end{equation}
and where $\Omega$ is a (non-empty) bounded, open and connected subset of $\R^d$ with a smooth
boundary $\Gamma$.
The set $Q=(0,T) \times \Omega$ is the time-space cylinder with lateral boundary 
$\Sigma = (0,T) \times \Gamma$ where $T>0$. 
More specifically, we are interested in the well-posedness of this Cauchy-Dirichlet problem in a suitable Sobolev space. 

While general results for the well-posedness of hyperbolic initial-boundary value problems are available, none of them can be just quoted here for an easy answer.
The problem with the MGT equation is that the spatial boundary $\Gamma$ is characteristic. Note that this third order equation does not contain a derivative of order three in normal direction, regardless of the geometry of the boundary.
Hence, the analysis by Sakamoto \cite{sakamoto82} cannot be used to study the
present IBVP.
Only H\"ormander's discussion of the mixed hyperbolic problem \cite[Section~12.9]{hormander83} allows for characteristic boundaries. But his approach is limited to the constant coefficient problem in the half space. A $C^\infty$-theory is presented, hence there are no estimates.

We explore two distinct approaches to the regularity analysis of the problem under consideration. 
One uses the theory of hyperbolic equations -- a proper perspective not embraced so far 
in the literature on the MGT equation -- and makes necessary adjustments to Sakamoto's approach because of the characteristic boundary.

The other one takes instead the perspective of Pandolfi and the first named author
which relates the MGT equation to a suitable wave equation with memory
\cite{bucci-pan-jee_2019}.
The latter approach then appeals to the hyperbolic regularity theory for linear wave equations with Dirichlet boundary conditions which is largely a consequence of Sakamoto's work on scalar strictly hyperbolic equations \cite{sakamoto_70,sakamoto82}.
An approach via first order system was given in the book by Chazarain and Piriou \cite[Example 7.3.10]{chaz-piriou_81}. 
A closer major reference for the present analysis is the work by Lasiecka, Lions and Triggiani, where energy methods and semigroup theory intertwine \cite{las-lions-trig_1986}.
The relevance of the theory of cosine operators for the regularity analysis of
second-order equations proves its strength also in the study of the third-order equation under examination.  

The two approaches combine to bring about the final conclusion, in the form of the interior and boundary regularity results stated as Theorem~\ref{t:main} below. It is interesting to note the differences between the two approaches. The approach based on the theory of semigroups requires a little more regularity of the Dirichlet boundary data to obtain optimal results concerning interior regularity. On the other hand, the hyperbolic approach assumes less regularity of the boundary data and produces less interior regularity. Both arguments lead to a trace regularity result.

\smallskip
The structure of the paper is readily outlined.
In the next subsection we present the full statement of the main results.
We leave a review of pertinent literature on the MGT equation  to Subsection~\ref{sub:literature}, along with the a brief mention of the physical considerations which brought about the (quasilinear) Jordan-Moore-Gibson-Thompson equation, and thus its linearization. 
This subsection is intended to provide a context for the MGT equation, along with an updated list of references. Its reading can be postponed, if one aims at focusing on the core of the present study.
Section~\ref{s:waves-w-memory} and Section~\ref{s:hyperbolic-systems}
contain the distinct analyses, that eventually culminate in the proof
of our main result, in accordance with the aforesaid distinct approaches.


\subsection{Main result} \label{sub:statements}
Consider the Cauchy-Dirichlet problem for the MGT equation, 
that we rewrite here for the reader's convenience:
\begin{equation}\label{e:ibvp-mgt_1}
\begin{cases}
w_{ttt}+\alpha w_{tt} -c^2 \Delta w-b \Delta w_t =f(t,x)
& \text{in $Q$}
\\[1mm]
w(0,\cdot)=w_0\,,\; w_t(0,\cdot)=w_1\,,\; w_{tt}(0,\cdot)=w_2
& \text{in $\Omega$}
\\[1mm]
w(t,x) =g(t,x) & \text{on $\Sigma$.}
\end{cases}
\end{equation}

\noindent
We briefly recall that the partial differential equation (PDE) referred to in the literature as the Moore-Gibson-Thompson equation -- in place of the longer 
Stokes-Moore-Gibson-Thompson-Jordan equation, which gives credit to various contributions during the decades, from Stokes \cite{stokes_1984} until Jordan \cite{jordan_2008,jordan_2014} -- is the linearization of a mathematical model of ultrasonic wave propagation, known as the Jordan-Moore-Gibson-Thompson equation; see the next subsection for an overview in a bit more detail.
The unknown $w=w(t,x)$, $(t,x)\in (0,T)\times \Omega$, 
represents the acoustic velocity potential or alternatively, the acoustic pressure
({\em cf.}~\cite{kalt-las-posp_2012} for a discussion on this issue).
The coefficients $c$, $b$, $\alpha$ are constant and positive; they represent the speed and diffusivity of sound ($c$, $b$), and a viscosity parameter ($\alpha$), respectively.
A relaxation parameter $\tau>0$ whose origin will appear clearer in 
Subsection~\ref{sub:literature} has been set equal to $1$, for simplicity of exposition.
The value 
\begin{equation}\label{e:gamma}
\gamma= \alpha - 
\frac{c^2}{b}
\end{equation}
will occur throughout, even though its property of being a
threshold for uniform stability will not play any role here;
see the former investigations of Kaltenbacher~et al.~\cite{kalten-etal_2011} and Marchand~et al.~\cite{marchand-etal_2012} (the latter providing a clarifying spectral analysis), as well as Dell'Oro and Pata \cite{delloro-pata_2017} (driven by the perspective of viscoelasticity).
Indeed, the first studies on (semigroup) well-posedness of the Cauchy-Dirichlet and Cauchy-Neumann problems associated with the MGT equation are carried out in \cite{kalten-etal_2011} and \cite{marchand-etal_2012}, in the case of homogeneous boundary conditions (i.e.~with $g\equiv 0$).
The key idea in both studies is the introduction of the auxiliary variable $z=w_t+c^2w/b$ and the equivalent coupled (PDE-ODE) system satisfied by $(z,w)$, where $z$ solves a second-order wave equation.

In the present work, focus is more specifically on the boundary-to-interior and interior-to-boundary regularity of the solutions, in a basic (and natural) functional setting. 
Accordingly, the regularity of the map
\begin{equation}
\big\{w_0,w_1,w_2,f,g\big\}\longrightarrow
\left\{w,w_t,w_{tt},\frac{\partial w}{\partial \nu},\frac{\partial^2 w}{\partial \nu^2}\right\}
\end{equation}
which associates to all data the interior solution (position, velocity, acceleration)
in $Q$, as well as normal derivatives of order one and two on $\Sigma$, will be the object of our investigation. 
Since the differential equation is of order three, the most natural setting from the viewpoint of hyperbolic PDE is to look for solutions $w\in H^2_{\rm loc}(Q)$ or 
$w\in C([0,T],H^2(Q))$. 

Our main result concerning the third order PDE under investigation, stated below, is similar to the interior and boundary regularity results that pertains to the Cauchy-Dirichlet problem for second order wave equations; 
see \cite{sakamoto82}, \cite[Theorem 3.4, Theorem 2.1]{las-lions-trig_1986}.

\begin{theorem} \label{t:main}
(a) With reference to the IBVP \eqref{e:ibvp-mgt_1} with $0<T<+\infty$, assume that
\begin{subequations}\label{e:hypo}
\begin{align}
& w_0\in H^2(\Omega)\,, \; w_1\in H^1(\Omega)\,, \; w_2\in L^2(\Omega)\,,
\label{e:hypo_1}
\\[1mm]
& f\in L^2(Q)=L^2(0,T;L^2(\Omega))\,,
\label{e:hypo_2}
\\[1mm]
& g\in C([0,T],H^{3/2}(\Gamma))\cap H^2(0,T;L^2(\Gamma))\,,
\label{e:hypo_3}
\\[1mm]
& g_t\in C([0,T],H^{1/2}(\Gamma))\,,
\label{e:hypo_4}
\end{align}
\end{subequations}
along with the compatibility conditions
\begin{equation}\label{e:cc}
w_0|_\Gamma= g|_{t=0}\in H^{3/2}(\Gamma)\,, \quad 
w_1|_\Gamma= g_t|_{t=0}\in H^{1/2}(\Gamma)\,.
\end{equation}
Then the unique solution to \eqref{e:ibvp-mgt_1} satisfies
\begin{equation}\label{e:interior-reg}
(w,w_t,w_{tt})\in C([0,T],H^2(\Omega)\times H^1(\Omega)\times L^2(\Omega))\,.
\end{equation}
If, in addition $g\in H^2(\Sigma)$, then we have the trace regularity result
\begin{equation}\label{e:boundary-reg}
 \frac{\partial^2 w}{\partial t\,\partial \nu} \in L^2(\Sigma)\,.
\end{equation}
 
\smallskip
\noindent
(b) Assume that $g$ satisfies -- in place of \eqref{e:hypo_3}-\eqref{e:hypo_4} -- the
weaker property 
\begin{equation}\label{e:stronger-hypo}
g\in H^1(\Sigma) \mbox{ and } g_t \in H^1(\Sigma)\,,
\end{equation}
with the regularity \eqref{e:hypo_2} of the affine term and 
\[
 w_0\in H^1(\Omega)\,,\; w_1\in H^1(\Omega)\,,\;w_2\in L^2(\Omega)
\]
(together with the compatibility conditions \eqref{e:cc}).
Then, the unique solution to \eqref{e:ibvp-mgt_1} satisfies
\[
 (w,w_t,w_{tt})\in C([0,T],H^1(\Omega)\times H^1(\Omega)\times L^2(\Omega))\,,
\] 
and the following trace regularity result holds true:
\begin{equation}\label{e:boundary-reg1}
\frac{\partial w}{\partial \nu}\,, \, \frac{\partial^2 w}{\partial t\partial \nu} \in L^2(\Sigma)\,.
\end{equation}
All quantities are continuous with respect to the data, consistently with the respective topologies.
\end{theorem}

\begin{remarks}[\sl Boundary regularity]
\begin{rm}
(i) It is important to emphasize that -- just like in the case of wave equations or other hyperbolic-like PDE -- the boundary regularity \eqref{e:boundary-reg} cannot be inferred on the only basis of the interior regularity 
$(w,w_t)\in C([0,T],H^2(\Omega)\times H^1(\Omega))$. 
Similarly, the regularity \eqref{e:boundary-reg1} does not follow directly from the interior regularity 
$(w,w_t)\in C([0,T],H^1(\Omega)\times H^1(\Omega) )$.
\\
(ii) The MGT equation does not yield the full trace regularity $\partial_\nu w \in H^1(\Sigma)$ and $\partial_\nu^2 w\in L^2(\Sigma)$ which would be expected from a 
third-order strictly hyperbolic PDE with non-characteristic boundary \cite{sakamoto82}. This is exactly because the boundary $\Sigma$ is characteristic. 
\\
(iii) While a trace regularity result for the second-order normal derivative cannot be obtained, we wonder whether $g\in H^2(\Sigma)$ implies $\partial w/\partial \nu \in H^1(\Sigma)$. 
\\
(iv) 
Instrumental to the study of an inverse problem, the regularity of boundary trace \eqref{e:boundary-reg} is shown in \cite{lecaros-etal_2020} for the special case of homogeneous boundary data.
The said result is contained in our Proposition 2.3 which pertains to the general non-homogeneous case, and {\sl a fortiori} in Theorem~\ref{t:main}, part~(a).
The respective proofs are distinct.
\end{rm}
\end{remarks}

\begin{remarks}[\sl Interior regularity]
\begin{rm}
In the case of homogeneous boundary data and affine term, the interior regularity of solutions \eqref{e:interior-reg} (in Theorem~\ref{t:main}, part (a)) 
is found already in the well-posedness result of \cite[Theorem~2.1]{marchand-etal_2012}, 
and also in the recent \cite[Theorem~5.3]{bucci-pan-jee_2019}.
Indeed, $g\equiv 0$ combined with the assumption \eqref{e:hypo_1} and the compatibility conditions \eqref{e:cc} yields $w_0\in H^2(\Omega)\cap H^1_0(\Omega)$ as well as 
$w_1\in H^1_0(\Omega)$, which gives 
$(w_0,w_1,w_2)\in [H^2(\Omega)\cap H^1_0(\Omega)]\times H^1_0(\Omega)\times L^2(\Omega)$, where the latter functional space is nothing but the space $U_3$ 
in the statement of the said Theorem~2.1 of \cite{marchand-etal_2012}.
And besides, the result is contained in the complex of regularity results summarized by the table~2 of \cite[Theorem~5.3]{bucci-pan-jee_2019}, if one takes in particular
$\lambda=2$, $\mu=1$, $\nu=0$.

As for part~(b) of Theorem~\ref{t:main}, we note that in the case $g\equiv 0$ it asserts in particular
the well-posedness of the IBVP for the MGT equation (with trivial boundary data) in the space 
$H^1_0(\Omega)\times H^1_0(\Omega)\times L^2(\Omega)$.
This result was originally proved in \cite{kalten-etal_2011}, with the significant specification of the group
property of the evolution.
It is also contained in the aforesaid 
\cite[Theorem~5.3]{bucci-pan-jee_2019}.
\end{rm}
\end{remarks}


\subsection{Background, literature review} \label{sub:literature}
The MGT equation arises in the context of a branch of physics and acoustics known as nonlinear acoustics (NLA), where one deals more specifically with sound waves of sufficiently large amplitudes.
The reader is referred e.g. to the monographs \cite{rudenko-soluyan-nla_1977} and
\cite{enflo-hedberg-nla_2006}, offering a predominant either physical or mathematical
treatment, respectively.
The review paper \cite{kalten-review_2015} provides an overview of established PDE models
of nonlinear sound propagation, as well as of more recent developements, 
along with a very useful collection of references.

Aiming to introduce a minimal mathematical background for the subject of the present investigation, we record explicitly two classical PDE models of NLA, namely, the Westervelt equation
\begin{equation*}
u_{tt}-c^2\Delta u - b \Delta u_t= \frac{\beta_a}{\rho c^2} (u^2)_{tt}
\qquad \text{in $(0,T)\times \Omega$}
\end{equation*}
(formulated in terms of the acoustic pressure $u$), 
and the Kuznetsov equation
\begin{equation*}
\psi_{tt}-c^2\Delta \psi - b \Delta \psi_t= 
\frac{\partial}{\partial t}\Big(\frac{\beta_a-1}{c^2}\psi_t^2+|\nabla \psi|^2\Big)
\qquad \text{in $(0,T)\times \Omega$}
\end{equation*}
(formulated in terms of the acoustic velocity potential $\psi$);
all constants that occurr in the equations are positive: 
$c, b$ are the speed and diffusivity of sound, $\rho>0$ is the mass density, $\beta_a>1$.   
%
The connections of the Westervelt and Kuznetsov (and Khokhlov-Zabolotskaya-Kuznetsov)
models with the Navier-Stokes and Euler compressible system is explored in the recent 
\cite{dekkers-rozanova_hal2019}.

A well-recognized issue which arises in the modeling of propagation of acoustic and thermal waves
is the {\em paradox of heat conduction}, namely, the incongruity 
between the infinite speed of propagation of a thermal disturbance with the principle of classical mechanics known as causality; see, e.g., \cite{jordan_2014} and its references.
Aiming to overcome this issue, the use of the (space-time) Maxwell-Cattaneo law
(\cite{cattaneo_1948,cattaneo_1958})
\begin{equation*}
\tau \dot{q} + q = -\kappa \nabla \theta
\end{equation*}
as constitutive relation for the heat flux $q$ (in place of the Fourier law, 
that corresponds to $\tau =0$, whilst above $\tau>0$)
has been proposed; such a choice eventually leads to the third-order PDE model
\begin{equation} \label{e:jmgt}
\tau \psi_{ttt} + \psi_{tt}-c^2\Delta \psi - b\Delta \psi_t=
\frac{\partial}{\partial t}\Big(\frac1{c^2}\frac{B}{2A}\psi^2_t+|\nabla \psi|^2\Big)
\end{equation}
($A, B$ are positive constants), known as the Jordan-Moore-Gibson-Thompson (JMGT) equation \cite{jordan_2008}.
The reader is referred to \cite{jordan_2014} for details on the derivation of 
the equation \eqref{e:jmgt}; see also \cite{kalten-review_2015}.

Notice that all three aforementioned equations are quasilinear PDE.
And yet, differently from the Westervelt and Kuznetsov equations,
the linearization of the JMGT equation -- i.e. the MGT equation -- is a strictly hyperbolic equation, as enlightened clearly in the Introduction; its mathematical analysis raises nontrivial issues, despite its being linear.
(Although perhaps unnecessary, we recall that the linearization of the Westervelt and Kuznetsov equations
-- {\sl viz.}~the strongly damped wave equation $u_{tt}-c^2\Delta u - b \Delta u_t= 0$ --
has a parabolic-like behaviour, as its dynamics is governed by an analytic semigroup.)

\smallskip
We recall as first the study of well-posedness and long-time behaviour for the JMGT equation (with time-dependent viscosity) carried out in \cite{kalt-las-posp_2012};
see also the recent \cite{racke-said_2019}.
Former contributions to the understanding of the analytical features of its linearization
are found in \cite{kalten-etal_2011} and \cite{marchand-etal_2012},
where major focus is placed on well-posedness of IBVP with homogenous (Dirichlet or Neumann) boundary data and on stability properties; in particular,
\cite{marchand-etal_2012} further provides a detailed spectral analysis.
While they establish well-posedness in more than one functional setting and show that the dynamics is governed by a strongly continuous {\em group}, these works disclose the crucial role of the parameters $b$ and $\gamma$ (defined by \eqref{e:gamma}) for well-posedness and uniform stability, respectively.
Indeed, in the case $b=0$ the associated initial-boundary value problems are ill-posed
\cite{kalten-etal_2011}.
Given $b>0$, then $\gamma$ must be positive, if one wants to ensure the property of {\em uniform} stability.
These findings have been revisited in \cite{delloro-pata_2017} within the history framework.

Most relevant to the present work is the study of the Cauchy-Dirichlet and Cauchy-Neumann problems with non-homogeneous boundary data performed in \cite{bucci-pan-jee_2019}, where a novel viewpoint and avenue of investigation is adopted. 
More precisely, the MGT equation is embedded in a family of wave equations with memory depending on a vectorial parameter.
This viewpoint enables the derivation of both interior and trace regularity results;
see \cite[Theorem~5.3]{bucci-pan-jee_2019} and \cite[Corollary~6.3]{bucci-pan-jee_2019}. 
The boundary data are assumed to be square integrable (in time and space) and provide a unique solution in a weaker topology, compared to our result.  

We recently learnt about the subsequent work 
\cite{trig_2020} that proves results about the same Cauchy-Dirichlet and Cauchy-Neumann problems, taking the original 
path of \cite{kalten-etal_2011} and \cite{marchand-etal_2012}.
Trace regularity results for both problems are obtained therein, in the case of
homogeneous initial data and forcing term, under square integrable boundary data which are subject to a continuity condition at $t=0$. 

\smallskip
Because high intensity focused ultrasound plays a central role in several medical procedures as well as industrial applications, optimal control problems arise naturally in the context of NLA.
A thorough overview of the literature on optimization problems associated with nonlinear PDE models for acoustic wave propagation is beyond the present work's scopes; {\sl cf.}~\cite{kalten-review_2015} and its references.
We limit ourselves to studies on the (MGT and) JMGT model. 
A functional-analytical framework and a solution to a minimization problem associated
with the MGT equation is provided in \cite{bucci-las_2019}, 
combining variational arguments with operator-theoretic techniques.
Specifically for the JMGT equation, we recall that \cite{nikolic-kalten_2017} deals with shape optimization; a sensitivity analysis with respect to the (relaxation) parameter $\tau>0$ has been carried out in \cite{kalten-nikolic_2019,kalten-nikolic_arxiv2019}.

We finally list a number of research works on the JMGT or the MGT equation, 
with a wealth of diverse focuses and goals.
These contributions concern:
the case $\gamma<0$, with an insight into the chaotic behaviour of the dynamics 
\cite{conejero-etal_2015},
inverse problems (\cite{liu-trig_2013,liu-trig_2014}, \cite{lecaros-etal_2020}),
long-time behaviour and attractors \cite{caixeta-etal_2016,caixeta-etal-attractors_2016},
explicit decay rates (\cite{pellicer-said_2019}, \cite{pellicer-sola_2019}),
null controllability for both the JMGT and MGT equations
\cite{lizama-etal_2019},
semilinear variants of the linear model and blow-ups \cite{chen-palmieri-dcds_2020}.
A variation of the original PDE model that displays an {\em additional} memory term
has been studied first in \cite{las-wang-part1_2016,las-wang-part2_2015}.
Questions that are explored and responded regard primarily well-posedness, the effect of the dissipation brought about by the memory term, decay rates.
Most recent articles include \cite{alves-etal_2018}, \cite{liu-etal_2019},
and \cite{delloro-las-pata_2016,delloro-las-pata_2019}.


\section{An approach exploiting the connection to wave equations with memory}
\label{s:waves-w-memory}
In this section we prove part (a) of Theorem \ref{t:main}.
Besides the interior regularity result we show a first boundary regularity result
(both stated therein); the latter is highlighted as Proposition~\ref{p:trace-partial}.
A sought-after improvement of the regularity of the boundary traces is briefly discussed in Remark~\ref{r:trace-partial}.

The basis of our proofs is the connection between the MGT equation and a suitable integro-differential equation devised in \cite{bucci-pan-jee_2019}.
The theory of Volterra equations (of the second kind) and the regularity theory
for second-order wave equations provide the tools.
Accordingly, for the reader's convenience we begin by considering the Cauchy-Dirichlet problem for second order wave equations: we recall the basic mathematical tools and notation, and the representation formula for its solutions that involves the cosine (and sine) operator, along with a few relevant regularity results.
The intermediate Proposition~\ref{p:volterra} establishes and details the connection between the MGT equation with the said Volterra equation, and constitutes an essential prerequisite for the understanding of the subsequent analysis. See also Remark~\ref{r:on-volterra}.

\subsection{Second order wave equations. Preliminaries}
Consider the following IBVP for a second order (linear) wave equation in the unknown
$z=z(t,x)$:
\begin{equation}\label{e:ibvp-wave}
\begin{cases}
z_{tt}=\Delta z +f & \text{in $Q$}
\\[1mm]
z(0,\cdot)=z_0\,, \; z_t(0,\cdot)=z_1  & \text{in $\Omega$}
\\[1mm]
z|_\Sigma=g\,.  &
\end{cases}
\end{equation}
Let $A$ be the realization of the Laplace operator in $L^2(\Omega)$, with Dirichlet
boundary conditions (BC); namely,
\begin{equation}\label{e:laplacian-realization}
A z:=\Delta z\,, \quad \cD(A)=H^2(\Omega)\cap H^1_0(\Omega)\,.
\end{equation}
It is well-known that the operator $A$, originally defined as in \eqref{e:laplacian-realization}, can be extended as 
$A\colon L^2(\Omega) \rightarrow [\cD(A^*)]'$.
Moreover, the fractional powers of $-A$ are well defined; {\sl cf.}~\cite[Vol.~II, \S~10.5.4]{las-trig-redbooks}
(paying attention to the fact that the present $A$ is denoted by $-\cA$ therein, whereas here $\cA$ is a different operator).
The Dirichlet (Green) map $D$ is defined as usual by
\begin{equation}\label{e:green-map}
D\colon L^2(\Gamma)\ni \varphi\longmapsto D\varphi=:\psi \;
\Longleftrightarrow \;
\begin{cases}
\Delta \psi =0 & \textrm{in $\Omega$}
\\[1mm]
\psi=\varphi & \textrm{on $\Gamma$}\,,
\end{cases}
\end{equation}
namely, $\psi=D \varphi$ is the harmonic extension of $\varphi$ from the boundary of
$\Omega$ into its interior.
Thus, the IBVP \eqref{e:ibvp-wave} corresponds to the abstract Cauchy problem
\begin{equation}\label{e:abstract-cauchy}
\begin{cases}
y'=\mathbb{A}y+\mathbb{B}g & \text{in $[\cD(\mathbb{A}^*)]'$}
\\[1mm]
y(0)=y_0
\end{cases}
\end{equation}
where we set $y(t)=(z(t),z_t(t))$, $y_0:=(z_0,z_1)$ and the linear operators $(\mathbb{A}$, $\mathbb{B})$ have the following explicit representation (in terms of $\cA$ and $D$), respectively:
\begin{equation*}
\mathbb{A}=\begin{pmatrix} 0 & I 
\\[1mm]
A & 0 \end{pmatrix}\,, \qquad
\mathbb{B}= \begin{pmatrix} 0 
\\[1mm]
-A D\end{pmatrix}\,;
\end{equation*}
in particular then, $\mathbb{B}\colon L^2(\Gamma) \longrightarrow [\cD(\mathbb{A}^*)]'$.
The operator $\mathbb{A}$ is the infinitesimal generator of a $C_0$-semigroup 
$e^{\mathbb{A}t}$, $t\ge 0$, e.g. on $Y=\cD((-A)^{1/2})\times L^2(\Omega)$.
The abstract differential formulation \eqref{e:abstract-cauchy} brings about the
following integral representation of the solution $y(t)$:
\begin{equation}\label{e:mild-sln}
y(t) = e^{\mathbb{A}t}y_0 + \int_0^t  e^{\mathbb{A}(t-s)}\mathbb{B}g(s)\,ds\,,
\end{equation}
which {\em a priori} makes sense at least on $[\cD(\mathbb{A}^*)]'$.

For cosine and sine operators we follow the notation 
adopted already in \cite[Section~2]{bucci-pan-jee_2019}.
Introduce the operator $\cA$ and the families of operators $R_+(\cdot)$, $R_{-}(\cdot)$
defined as follows:
\begin{equation}\label{e:cosine-sine}
\cA=i(-A)^{1/2}\,,\qquad 
R_+(t)=\frac{e^{\cA t}+e^{-\cA t}}{2}\,,
\qquad R_-(t)=\frac{e^{\cA t}-e^{-\cA t}}{2}\,.
\end{equation}
$R_+(t)$ is the strongly continuous {\em cosine} operator generated by $-A$ in
$L^2(\Omega)$ (\cite{sova_1966}, \cite{fattorini_1985}).
(See also \cite[Vol.~II, \S\,10.5.4]{las-trig-redbooks}, paying attention to the
distinct notations: indeed, even though the cosine operator $R_+(t)$ coincides with $C(t)$ therein, the operator denoted by $S(t)$ is actually the present $\cA^{-1}R_-(t)$.) 
We note that $\cA$ is the infinitesimal generator of a $C_0$-{\em group} of operators in $L^2(\Omega)$.
It is by now well-known that the semigroup $e^{\mathbb{A}t}$ admits the explicit representation
\begin{equation*}
e^{\mathbb{A}t}=
\begin{pmatrix}
R_+(t) & \cA^{-1} R_-(t)
\\[1mm]
-\cA R_-(t) & R_+(t)
\end{pmatrix}\,,
\end{equation*}
in terms of the cosine and sine operators, and that the solution to the IBVP 
\eqref{e:ibvp-wave} is given by
\begin{equation}\label{e:waves-explicit}
\begin{split}
z(t)& =R_+(t)z_0+\cA^{-1} R_-(t)z_1 +\cA^{-1} \int_0^t R_-(t-s)f(s)\,ds
\\[1mm]\
& \myspace -\cA\int_0^t R_-(t-s)D g(s)\,ds\,;
\end{split}
\end{equation}
see, e.g., \cite[Vol.~II, \S~10.5.4]{las-trig-redbooks}.

Since the regularity pertaining to the cosine and sine operators, specifically when
acting either on elements of the Banach space $L^2(\Omega)$ or (on elements) of the domain of the generator will be used repeatedly, we record here first of all that
\begin{equation}\label{e:semigroups}
\begin{split}
R_+(\cdot)x\,, \;R_-(\cdot)z\in C([0,T],L^2(\Omega)) \quad \text{with $x,z\in L^2(\Omega)$,}
\\[1mm]
R_+(\cdot)x\,, \;R_-(\cdot)z\in C([0,T],\cD(A)) \quad \text{with $x,z\in \cD(A)$\,.}
\end{split}
\end{equation}
Furthermore, a rewriting (in abstract form) of the nontrivial result
\begin{equation*}
\begin{cases}
(z_0, z_1)\in L^2(\Omega)\times H^{-1}(\Omega)\,, \; f\equiv 0
\\[1mm]
g\in L^2(\Sigma)=L^2(0,T;L^2(\Gamma)) 
\end{cases}
\Longrightarrow (z,z_t,z_{tt})
\in C([0,T],L^2(\Omega)\times H^{-1}(\Omega))\times H^{-2}(\Omega))
\end{equation*}
gives in particular
\begin{equation}\label{e:pre-crucial}
g\in L^2(\Sigma) 
\Longrightarrow 
\begin{pmatrix} 
\cA\int_0^\cdot R_-(\cdot-s) Dg(s)\,ds
\\[1mm]
\cA^2\int_0^\cdot R_+(\cdot-s) Dg(s)\,ds
\end{pmatrix}
\in C([0,T],L^2(\Omega)\times H^{-1}(\Omega))\,,
\end{equation}
(for $z_0=z_1\equiv 0$), which in turn tells us
\begin{equation} \label{e:crucial}
g\in L^2(\Sigma) 
\Longrightarrow 
\begin{pmatrix} 
\cA\int_0^\cdot R_-(\cdot-s) Dg(s)\,ds
\\[1mm]
\cA\int_0^\cdot R_+(\cdot-s) Dg(s)\,ds
\end{pmatrix}
\in C([0,T],L^2(\Omega)\times L^2(\Omega))\,,
\end{equation}
since $H^{-1}(\Omega)
\equiv [\cD(\cA)]'$.
The 
boundary-to-interior regularity results expressed by \eqref{e:crucial} will be used
quite often in the computations.
(We note that the result mentioned as first, below \eqref{e:semigroups},
has been shown to follow from a (sharp) boundary regularity estimate for the solution
to a dual problem.
Both aforesaid results are recorded in \cite{las-trig-redbooks} as Theorem~10.5.3.2 and Theorem~10.5.3.1, respectively.
Definitive proofs of these results are originally devised in the contemporary
\cite{lions-book_1983} and \cite{las-trig-hyperbolic_1983}; a roadmap for the subject
can be found in \cite[Section~10.5]{las-trig-redbooks} and in the relative Notes 
\cite[p.~1060]{las-trig-redbooks}.
We recall explicitly \cite{sakamoto_70} and \cite{lions-mag_72} as the former contributions to the regularity analysis of the IBVP \eqref{e:ibvp-wave}.)

\smallskip
It is useful to introduce the symbol $\cK$ to denote the linear operator defined by  
\begin{equation}\label{e:affine-convoluted}
\cK\colon f \longrightarrow (\cK f)(t):= \cA^{-1} \int_0^t R_-(t-s) f(s) \,ds\,.
\end{equation}
Specifically when $f(\cdot)=Dg(\cdot)$, one has
\begin{equation*}
(\cK Dg(\cdot))(t):= \cA^{-1} \int_0^t R_-(t-s) Dg(s) \,ds
=\cA^{-2} \Big[\cA\int_0^t R_-(t-s) Dg(s) \,ds\Big]\,,
\end{equation*}
which tells us -- in view of \eqref{e:crucial} -- that $\cK Dg\in \cD(\cA^2)\equiv \cD(A)$, if $g\in L^2(\Sigma)$. 


\subsection{Proof of Theorem~\ref{t:main}, part (a), and a first trace regularity result}
\label{sub:proof-franb}
The starting point of the present analysis is the connection between the MGT equation
and a wave equation with memory (and affine term depending also on initial data) first,
and a suitable Volterra equation of the second kind next, to which problem
\eqref{e:ibvp-mgt_1} can be reduced to, as shown in \cite{bucci-pan-jee_2019}.
The Proposition below embodies the statement of Proposition~3.6 in 
\cite{bucci-pan-jee_2019}, without neglecting the nontrivial forcing term $f$
in the original IBVP \eqref{e:ibvp-mgt_1}, and removing the translation of the differential operator $\Delta$ to $\Delta-I$, which is here unnecessary.
The symbol $\ast$ below denotes the usual convolution operation.


\begin{proposition} \label{p:volterra}
Any solution $w$ to the initial/boundary value problem \eqref{e:ibvp-mgt_1} is such that
$v=e^{\frac{\gamma}{2}t}w$ solves the following Volterra equation of the second kind:
\begin{equation} \label{e:volterra} 
v(t)+\displaystyle\int_0^t L(t-s)v(s)\,ds=H(t) \qquad \text{in $Q$}
\end{equation}
where $L(\cdot)$ is the strongly continuous kernel defined by
\begin{equation} \label{e:super-kernel}
L(t)v=-\frac{\beta}{\sqrt b} \cA^{-1}R_-(\sqrt b t) v
-\frac{1}{\sqrt b}\cA^{-1}\int_0^t R_-(\sqrt b(t-s)) K(s)v\,ds\,,
\end{equation}
and the affine term $H(\cdot)$ is given -- in terms of the initial and boundary data
-- by
\begin{equation} \label{e:affine}
\begin{split}
H(t)&=\Big[R_+(\sqrt{b}t)+\frac{\gamma}{2\sqrt{b}}\cA^{-1}R_-(\sqrt{b}t)\Big]w_0
+\frac{1}{\sqrt{b}}\cA^{-1} R_-(\sqrt{b}t) w_1+
\\[1mm]
& \quad +\frac{1}{\sqrt{b}} \cA^{-1}\int_0^t R_-(\sqrt{b}(t-s))
\big[h_0(t) w_0+h_1(t)w_1+h_2(t)(w_2-\Delta w_0)\big]\, ds-
\\[1mm]
& \quad +\frac{1}{\sqrt{b}} \cA^{-1}\int_0^t R_-(\sqrt{b}(t-s))\tilde{f}(s)\,ds
-\sqrt{b} \cA\int_0^t R_-(\sqrt{b}(t-s)) D\tilde{g}(s)\,ds\,.
\end{split}
\end{equation}
In the above formulas $\tilde{g}= e^{\frac{\gamma}{2}t}g$, while
the constant $\beta$ 
and the functions $K(\cdot)$, $h_i(\cdot)$, $i=0,1,2$, $\tilde{f}(\cdot)$ read 
explicitly as follow:
\begin{equation}\label{e:various-functions}
\begin{split}
& \beta=-\gamma\Big(\frac{3}{4}\gamma -\alpha\Big)\,, 
\qquad
K(t) = -\gamma (\gamma-\alpha)^2 e^{(\frac{3}{2}\gamma-\alpha)t}\,,
\\
& h_0(t)=-\gamma (\gamma-\alpha) e^{(\frac{3}{2}\gamma -\alpha)t}\,,
\quad 
h_1(t)=-\gamma e^{(\frac{3}{2}\gamma -\alpha)t}\,,
\quad
h_2(t)= e^{(\frac{3}{2}\gamma-\alpha)t}\,;
\\
& 
\tilde{f}(t)= e^{\frac{\gamma}{2}t}\Big(\lambda(t)
+ \gamma\big(e^{-\frac{c^2}{b}\cdot}\ast \lambda\big)(t)\Big)\,, 
\quad 
\lambda(t):=\int_0^t e^{-\alpha (t-s)}f(s)\,ds\,.
\end{split}
\end{equation}
In particular then, the initial and boundary data for $v$ are related to those of $w$ as follows:
\begin{equation}\label{e:data-for-v}
\begin{split}
& v|_{t=0}=:v_0\equiv w_0\,, \quad v_t|_{t=0}=:v_1=\frac{\gamma}{2}w_0+w_1\,;
\\[1mm]
& v|_\Sigma =\tilde{g}:=e^{\frac{\gamma}{2}t}g\,.
\end{split}
\end{equation}

\end{proposition}

\begin{proof}
It suffices to follow the perspective of \cite{bucci-pan-jee_2019}, and more specifically
the arguments in Section~3.1, which eventually result in Proposition~3.6 therein.
With slight modifications and a straightforward computation, one arrives here at the
following (equivalent) IBVP for an integro-differential equation satisfied by 
$v=e^{\frac{\gamma}{2}t}w$:
\begin{equation}\label{e:ibvp-for-v}
\begin{cases}
v_{tt}=b \Delta v + \displaystyle\int_0^t K(t-s)v(s)\,ds + \beta v \,+
\\[2mm]
\myspace + \big[h_0(t) w_0+h_1(t)w_1+h_2(t)(w_2-b\Delta w_0)\big] 
+ \tilde{f}(t) & \text{in $Q$}
\\[2mm]
v(0,\cdot)=w_0\,, \; v_t(0,\cdot)=\frac{\gamma}{2}w_0+w_1 & \text{in $\Omega$}
\\[1mm]
v=\tilde{g} & \text{on $\Sigma$}
\end{cases}
\end{equation}
(with all functions defined in \eqref{e:various-functions} and 
$\tilde{g}=e^{\frac{\gamma}{2}t} g$). 
Thus, {\em mutatis mutandis}, the representation formula \eqref{e:waves-explicit} provides the tool.
\end{proof}

\smallskip
Before starting the proof of our main result, a few considerations on the regularity
of solutions to Volterra equations are in order; partly were given already in 
\cite{bucci-pan-jee_2019}, partly are new.
Consider a general Volterra equation of the second kind, that is $v+L\ast v=h$ with
\begin{equation*}
(L \ast v)(t)=\int_0^t L(t-s)v(s)\,ds=\int_0^t L(s)v(t-s)\,ds\,.
\end{equation*}
It is known that if $L$ is a strongly continuous function of time, with values in 
$\cL(\cH)$ ($\cH$ is a Hilbert space) and $h(\cdot)$ is an integrable $\cH$-valued function,
then the corresponding solution has the explicit representation 
\begin{equation}\label{e:sln-to-volterra}
v=h+\sum _{k=1}^{\infty} (-1)^k L^{(\ast k)}\ast h\,,
\end{equation}
where $L^{(\ast n)}$ indicate the iterated convolutions, recursively defined as
follow:
\begin{equation*}
L^{(\ast 1)}=L\,,\quad  L^{(\ast(n+1))}\ast h=L\ast \big(L^{(\ast n)}\ast h\big)\,,
\qquad n\ge 1\,;
\end{equation*}
the uniform convergence of the series can be easily proved 
({\sl cf.}~e.g.~\cite[Chapter~5]{corduneanu}).
From formula \eqref{e:sln-to-volterra} we see that the regularity in time and space
of $v$ is determined by the one of $h$, as well as of the said iterated convolutions.
In the present case, with $L$ given by \eqref{e:super-kernel} (and $h$ eventually
replaced by $H$), the first iteration reads as
\begin{equation*}
\begin{split}
(L \ast h)(t)&=-\frac{\beta}{\sqrt{b}} \cA^{-1}\int_0^t R_-(\sqrt{b}(t-s)) h(s)\,ds
\\[1mm]
& \qquad 
-\frac{1}{\sqrt{b}}\cA^{-1}\int_0^t 
\Big[\int_0^{t-s} R_-(\sqrt{b}(t-s-r)) K(r)\,dr\Big]\,h(s)\,ds\,.
\end{split}
\end{equation*}
Thus, in the regularity analysis of this convolution the properties of the first summand will prevail.
Since the said term is (neglecting the constant in front) nothing but $(\cK \,h)(\cdot)$, 
with the operator $\cK$ defined in \eqref{e:affine-convoluted},
the regularity properties of $\cK$ such as e.g.
\begin{equation}\label{e:calK-improves}
\cK\in \cL(L^1(0,T;L^2(\Omega)),C([0,T],H^1_0(\Omega))) 
\end{equation}
will imply that e.g.,
\begin{equation*}
h\in L^1(0,T;L^2(\Omega)) \Longrightarrow 
L \ast h\in C([0,T],H^1_0(\Omega))\,, 
\end{equation*} 
continuously with respect to the topologies under consideration. 
The same consideration is valid, {\em a fortiori}, to the next iterated convolution,
and so on.
In conclusion, since $\cK$ is smoothing both in time and space, it will suffice to pinpoint the regularity of the affine term $h$, which determines the regularity of $v$. 

\begin{remark}\label{r:on-volterra}
\begin{rm}
To summarize the above considerations: the representation \eqref{e:sln-to-volterra}, combined with the specific structure \eqref{e:super-kernel} of the kernel $L$ in the Volterra equation $v+L\ast v=h$, guarantees $v\in C([0,T],X)$ on the basis of 
$h\in C([0,T],X)$ also in the case $X$ is a Sobolev space $H^k(\Omega)$;
namely, also in the case $X$ is not the domain of a fractional power of the operator $-A$, as are the spaces $X_\alpha$ in the statement of Lemma~4.1 in \cite{bucci-pan-jee_2019}.
\end{rm}
\end{remark}


\smallskip
\noindent
\paragraph{\em Proof of Theorem~\ref{t:main}, part (a)}
In order to establish the interior regularity of solutions to the 
Cauchy-Dirichlet problem \eqref{e:ibvp-mgt_1} as specified by \eqref{e:interior-reg}, 
in view of Proposition~\ref{p:volterra} we turn our attention to the Volterra equation
\eqref{e:volterra}, that is $v+L\ast v=H$ (in short), with $L$ and $H$ defined
by \eqref{e:super-kernel} and \eqref{e:affine}, respectively.
Because specifically $v\in C([0,T],H^2(\Omega))$ will be guaranteed by $H\in C([0,T],H^2(\Omega))$ ({\sl cf.}~Remark~\ref{r:on-volterra}), the optimal regularity (in time and space) of the affine term $H(t)$ must be pinpointed.

\vspace{1mm}
\noindent
{\bf 1.}
In order to render the computations more readable, we initially set 
\begin{equation} \label{e:calf}
\cF(t) :=h_0(t)w_0+h_1(t)w_1+h_2 (t)\big(w_2-\Delta w_0\big)\,;
\end{equation}
one needs to remember that $\cF$ depends on initial data. 
Owing to the assumption \eqref{e:hypo_3}, which yields as well
$\tilde{g}\in H^2(0,T;L^2(\Gamma))$ for $\tilde{g}= e^{\frac{\gamma}{2}t}g$, we are allowed
to integrate by parts (in time) twice in \eqref{e:affine}, thereby obtaining (after setting
$b=1$ for the sake of simplicity and the readers' convenience) first
\begin{equation} \label{e:affine-rewrite_1}
\begin{split}
H(t)&=\Big[R_+(t)+\frac{\gamma}{2}\cA^{-1}R_-(t)\Big]w_0
+\cA^{-1} R_-(t) w_1+ R_+(t-s) D\tilde{g}(s)\Big|_{s=0}^{s=t}
\\[1mm]
& \quad 
-\int_0^t R_+(t-s) D\tilde{g}_t(s)\,ds\,
+ \cA^{-1}\int_0^t R_-(t-s)\big(\cF(s)+\tilde{f}(s)\big)\, ds 
\\[2mm]
&=R_+(t)\big(w_0-D\tilde{g}(0)\big)+\cA^{-1}R_-(t)\Big[\frac{\gamma}{2}w_0+ w_1\Big]
+ D\tilde{g}(t) 
\\[1mm]
& \quad -\int_0^t R_+(t-s) D\tilde{g}_t(s)\,ds\,
+\cA^{-1}\int_0^t R_-(t-s)\big(\cF(s)+ \tilde{f}(s)\big)\,ds\,,
\end{split}
\end{equation}
and next
\begin{equation} \label{e:affine-rewrite_2}
\begin{split}
H(t)&=R_+(t)\big[w_0-Dg(0)\big]+ \cA^{-1}R_-(t)\Big[\frac{\gamma}{2}w_0+ w_1\Big]
+ D\tilde{g}(t)+\cA^{-1}D\tilde{g}_t(t) 
\\[1mm]
& \quad 
-\cA^{-1}R_{-}(t)D\tilde{g}_t(0)
- \cA^{-1}\int_0^t R_-(t-s) D\tilde{g}_{tt}(s)\,ds 
\\[1mm]
& \quad 
+\cA^{-1}\int_0^t R_-(t-s)\big(\cF(s)+\tilde{f}(s)\big)\,ds 
\\[1mm]
& =R_+(t)\big[w_0-Dg(0)\big] + \cA^{-1} R_-(t)\Big[\frac{\gamma}{2}w_0+w_1
-D\Big(\frac{\gamma}{2}g(0)+g_t(0)\Big)\Big] 
\\[1mm]
& \quad 
+ D\tilde{g}(t)+\cA^{-1}D\tilde{g}_t(t)- \cA^{-1}\int_0^t R_-(t-s) D\tilde{g}_{tt}(s)\,ds\, 
\\[1mm]
& \quad 
+\cA^{-1}\int_0^t R_-(t-s)
\big[h_0(s)w_0+h_1(s)w_1+h_2 (s)\big(w_2-\Delta w_0\big)+\tilde{f}(s)\big]\,ds 
\\[1mm]
& 
=:\sum_{i=1}^9T_i\,.
\end{split}
\end{equation}
Let us examine either summand $T_i$, $i=1,2,\dots,9$ in \eqref{e:affine-rewrite_2}.
Recall the assumptions $w_0\in H^2(\Omega)$, $w_1\in H^1(\Omega)$, along with the definition \eqref{e:green-map} of the Dirichlet map $D$;
it will be used that 
\begin{equation} \label{e:improve-one-half}
D\in \cL(H^s(\Gamma),H^{s+1/2}(\Omega)) \qquad \forall s\,.  
\end{equation}

\noindent
Then, with $g(0)\in C([0,T],H^{3/2}(\Gamma))$ and $g_t(0)\in C([0,T],H^{1/2}(\Gamma))$, we see that 
$w_0-Dg(0)\in H^2(\Omega)$ and $w_1-Dg_t(0)\in H^1(\Omega)$, respectively;
in addition, on account of the compatibility conditions \eqref{e:cc}, we also know that
\begin{equation*}
\big[w_0-Dg(0)\big]\big|_{\Gamma}=0\,, \qquad \big[w_1-Dg_t(0)\big]\big|_\Gamma=0\,.
\end{equation*}
The above means in particular $w_0-Dg(0)\in \cD(A)$, so that the action of the cosine operator $R_+(t)$ generated by $A$ ensures that 
$T_1:=R_+(\cdot)\big[w_0-Dg(0)\big]\in C([0,T],\cD(A))$ as well.
Similarly, we have
\begin{equation*}
T_2:=\cA^{-1} R_-(\cdot)\Big[\frac{\gamma}{2}\big(w_0-Dg(0)\big)+\big(w_1-Dg_t(0)\big)\Big]
\in C([0,T],\cD(A))
\end{equation*}
and conclude that
\begin{equation}\label{e:terms_12}
T_i\in C([0,T],\cD(A))\subset C([0,T],H^2(\Omega))\,, \qquad i=1,2\,.
\end{equation} 
In view of the regularity assumptions \eqref{e:hypo_3}-\eqref{e:hypo_4} on the boundary datum, and taking into account once again \eqref{e:improve-one-half}, one finds
\begin{equation}\label{e:terms_34}
T_3+T_4:=D\tilde{g}+\cA^{-1}D\tilde{g}_t\in C([0,T],H^2(\Omega))\,.
\end{equation} 
As for $T_5$, readily
\begin{equation}\label{e:term_5}
T_5=-\cA^{-1}\int_0^t R_-(t-s) D\tilde{g}_{tt}(s)\,ds
= \cA^{-2}\Big\{-\cA\,\int_0^t R_-(t-s) D\tilde{g}_{tt}(s)\,ds\Big\}
\in C([0,T],\cD(A))
\end{equation}
as a consequence of the regularity result \eqref{e:crucial}, which holds true since
$\tilde{g}_{tt}\in L^2(\Sigma)$ (the role of $g$ is played by $\tilde{g}_{tt}$, here). 

Of the three summands 
$T_6=\big(\cK (h_0(\cdot)w_0\big)(t)$ and $T_7=\big(\cK (h_1(\cdot)w_1\big)(t)$  
and 
\begin{equation*}
T_8=\cA^{-1}\int_0^t R_-(t-s)h_2 (s)\big(w_2-\Delta w_0\big)\,ds
=\big(\cK (h_2(\cdot)\big(w_2-\Delta w_0)\big)(t)\,,
\end{equation*}
it is sufficient to analyze the latter. 
We simply have $w_2-\Delta w_0\in L^2(\Omega)$ by Assumption \eqref{e:hypo_1},
and \eqref{e:calK-improves} does not suffice to obtain the needed space regularity.
Then, we use the fact that $h_2\in C^\infty$, integrate by parts (in time),
and establish
\begin{equation}\label{e:term_8}
\begin{split}
T_8 &=-\cA^{-2}R_+(t-s)h_2 (s)\big(w_2-\Delta w_0\big)\Big|_{s=0}^{s=t}
+ \cA^{-2}\int_0^t R_+(t-s)h_2'(s)\big(w_2-\Delta w_0\big)\,ds
\\[1mm]
&=-\cA^{-2}\Big\{h_2(t)\big(w_2-\Delta w_0\big)-R_+(t)h_2(0)\big(w_2-\Delta w_0\big)
\\[1mm]
& \myspace - \int_0^t h_2'(t-s) R_+(s)\big(w_2-\Delta w_0\big)\,ds\Big\}
\in C([0,T],\cD(A))\,.
\end{split}
\end{equation}
That 
\begin{equation}\label{e:terms_67}
T_6\,, \, T_7\in C([0,T],\cD(A)) 
\end{equation}
follows similarly (neglecting the better regularity in space, as it is here unnecessary).

To analyze the term
\begin{equation*}
T_9:=-\cA^{-1}\int_0^t R_-(t-s)\tilde{f}(s)\,ds
\end{equation*}
that is $T_9=\big(\cK \tilde{f}(\cdot)\big)(t)$, we observe preliminarly that the regularity of $\tilde{f}(t)$ is determined by the one of $\lambda(t)$, which in turn
is the same of $\big(e^{-\alpha \cdot}\ast f\big)(t)$.
Since by assumption $f\in L^2(Q)$, then $\lambda\in H^1(0,T;L^2(\Omega))$, with
\begin{equation*}
\lambda'(t) = \frac{d}{dt}\int_0^t e^{-\alpha(t-s)}f(s)\,ds=
f(t) - \alpha \lambda(t)\,.
\end{equation*}
Therefore, we find
\begin{equation*}
\begin{split}
& \cA^{-1}\int_0^t R_-(t-s)\lambda(s)\,ds=
-\cA^{-2}\Big\{R_+(t-s)\lambda(s)\Big|_{s=0}^{s=t}
- \int_0^t R_+(t-s)\big(f(s) - \alpha \lambda(s)\big)\,ds\Big\} 
\\[1mm]
&=-\cA^{-2}\Big\{\lambda (t) 
- \int_0^t R_+(t-s)\big(f(s) - \alpha \lambda(s)\big)\,ds\Big\} 
\\[1mm]
&=-\cA^{-2}\Big\{\lambda (t)-\int_0^t R_+(t-s)\big(f(s) - \alpha \lambda(s)\big)\,ds\Big\}
\in C([0,T],\cD(A))\,,
\end{split}
\end{equation*}
which gives 
\begin{equation}\label{e:term_9}
T_9\in C([0,T],\cD(A))\subset C([0,T],H^2(\Omega))
\end{equation}
as well.
In conclusion, combining \eqref{e:term_9} with \eqref{e:term_8}, \eqref{e:terms_67}, \eqref{e:term_5}, \eqref{e:terms_34} and \eqref{e:terms_12} we find 
$$H\in C([0,T],H^2(\Omega))\,,$$
which implies $v\in C([0,T],H^2(\Omega))$; consequently, we attained 
$w\in C([0,T],H^2(\Omega))$, that is the regularity statement for the 
`position' in \eqref{e:interior-reg}.

\vspace{1mm}
\noindent
{\bf 2.}
We need to show now that regularity pertaining to $w_t$ and $w_{tt}$ in 
\eqref{e:interior-reg} holds true, namely, that 
$(w_t,w_{tt})\in C([0,T],H^1(\Omega)\times L^2(\Omega))$ holds.
We proceed similarly as in the proof of \cite[Theorem~4.2]{bucci-pan-jee_2019}:
the said regularity will be inherited by the regularity of the derivatives $v_t$ and $v_{tt}$
of the solution $v$ to the Volterra equation \eqref{e:volterra}, which is in turn determined by the one of the respective right hand sides of the Volterra equations satisfied by $v_t$ and $v_{tt}$.
Rewrite \eqref{e:volterra} as 
\begin{equation*} 
v(t)+\displaystyle\int_0^t L(s)v(t-s)\,ds=H(t)
\end{equation*}
to deduce that $v_t$ satisfies
\begin{equation}\label{e:volterra-v_t} 
v_t(t)+ \displaystyle\int_0^t L(t-s)v_t(s)\,ds=H_t(t)-L(t) v_0\,,
\end{equation}
where -- given the expression \eqref{e:super-kernel} of the kernel $L$ -- we know
that 
$$L(t) v_0 =L(t) w_0\in C([0,T],\cD(\cA))\subset C([0,T],H^1(\Omega))\,.$$
Therefore, aiming at showing the regularity statement for the `velocity' $w_t$ in 
\eqref{e:interior-reg}, and given the right hand side in the Volterra equation
\eqref{e:volterra-v_t}, the regularity $H_t\in C([0,T],H^1(\Omega))$ is what we
need to prove.
Restart from \eqref{e:affine-rewrite_1}, to find -- after an integration by parts (in time) in two summands --
\begin{align*}
H_t(t)&=\cA R_-(t)\big[w_0-Dg(0)\big]+R_+(t)\Big[\frac{\gamma}{2}w_0+ w_1\Big]
\nonumber\\[1mm]
& \quad -\cA \int_0^t R_-(t-s) D\tilde{g}_t(s)\,ds\, 
+ \int_0^t R_+(t-s)\big[\cF(s)+ \tilde{f}(s)\big]\,ds 
\nonumber\\[1mm]
& =\cA R_-(t)\big[w_0-Dg(0)\big]+R_+(t)\Big[\frac{\gamma}{2}w_0+ w_1\Big]
+R_+(t-s)D\tilde{g}_t(s)\Big|_{s=0}^{s=t} 
\nonumber\\[1mm]
& \quad -\int_0^t R_+(t-s) D\tilde{g}_{tt}(s)\,ds  
- \cA^{-1} R_-(t-s)\big[\cF(s)+ \tilde{f}(s)\big]\Big|_{s=0}^{s=t}
\nonumber\\[1mm]
& \quad 
+ \cA^{-1}\int_0^t R_-(t-s)\big[\cF'(s)+ \tilde{f}'(s)\big]\,ds 
\nonumber\\[1mm]
& = \cA R_-(t)\big[w_0-Dg(0)\big]+R_+(t)\Big[\frac{\gamma}{2}w_0+ w_1\Big]
+D\tilde{g}_t(t)- R_+(t)D\tilde{g}_t(0) 
\nonumber\\[1mm]
& \quad -\int_0^t R_+(t-s) D\tilde{g}_{tt}(s)\,ds + \cA^{-1} R_-(t)\cF(0)
\nonumber\\[1mm]
& \quad
+ \cA^{-1}\int_0^t R_-(t-s)\big[\cF'(s)+\tilde{f}'(s)\big]\,ds\,,
\end{align*}
where it has been used that $\tilde{f}(0)=0$ and once again that $R_-(0)=0$
(the term $D\tilde{g}_t(t)$ and its opposite are canceled at the outset).
Substituting $\tilde{g}_t(0)=\frac{\gamma}{2}g(0)+ g_t(0)$ in the latter espression,
we arrive at the following clean representation of $H_t(t)$:
\begin{equation}\label{e:derive-affine_1bis}
\begin{split}
H_t(t)&=\cA R_-(t)\big[w_0-Dg(0)\big]
+R_+(t)\Big[\frac{\gamma}{2}\big(w_0-Dg(0)\big)+\big(w_1-Dg_t(0)\big)\Big]
+D\tilde{g}_t(t) 
\\[1mm]
& \quad -\int_0^t R_+(t-s) D\tilde{g}_{tt}(s)\,ds + \cA^{-1} R_-(t)\cF(0) 
\\[1mm]
& \quad
+ \cA^{-1}\int_0^t R_-(t-s)\big[\cF'(s)+ \tilde{f}'(s)\big]\,ds\,.
\end{split}
\end{equation}
We examine now each summand in the right hand side of \eqref{e:derive-affine_1bis} to find, in succession,
\begin{equation*}
\begin{split}
& \cA R_-(\cdot)\big[w_0-Dg(0)\big]\in C([0,T],\cD(\cA))\subset C([0,T],H^1(\Omega))\,,
\\[2mm]
& R_+(\cdot)\Big[\frac{\gamma}{2}\big(w_0-Dg(0)\big)+\big(w_1-Dg_t(0)\big)\Big]
\in C([0,T],\cD(\cA))\,,
\\[2mm]
& D\tilde{g}_t\in C([0,T],H^1(\Omega))\,, \quad \text{in view of} \; 
\tilde{g}_t\in C([0,T],H^{1/2}(\Gamma))\,,
\\[2mm]
& \int_0^t R_+(t-s) D\tilde{g}_{tt}(s)\,ds\in C([0,T],\cD(\cA))\,,
\quad \text{in view of \eqref{e:crucial}}
\\[2mm]
& \cA^{-1} R_-(\cdot)\cF(0)\in C([0,T],\cD(\cA)) \quad \text{because of} \; 
R_-(\cdot)\cF(0)\in C([0,T],L^2(\Omega))\,,
\\[2mm]
& \cA^{-1}\int_0^\cdot R_-(\cdot-s) \big[\cF'(s)+ \tilde{f}'(s)\big]\,ds
\in C([0,T],\cD(\cA))\,,
\end{split}
\end{equation*}
that confirms the membership 
\begin{equation*}
H_t\in C([0,T],H^1(\Omega))\,.
\end{equation*}
Consequently, the above establishes $v_t\in C([0,T],H^1(\Omega))$, so that 
$w_t\in C([0,T],H^1(\Omega))$, as required by \eqref{e:interior-reg}.
It remains to prove $w_{tt}\in C([0,T],L^2(\Omega))$, which is accomplished in the
next step.

\vspace{1mm}
\noindent
{\bf 3.}
From \eqref{e:volterra-v_t} it follows that $v_{tt}$ solves the Volterra equation
\begin{equation}\label{e:volterra-v_tt} 
v_{tt}+ \displaystyle\int_0^t L(t-s)v_{tt}(s)\,ds=H_{tt}(t)-\frac{d}{dt} \Big[L(t) v_0\Big]-L(t)v_1\,,
\end{equation}
where it is immediately seen that
\begin{equation}\label{e:trivial}
\frac{d}{dt}\Big[L(t) v_0\Big]-L(t)v_1 
= -\beta R_+(t)v_0 - \int_0^t R_+(t-s)K(s)v_0\,ds - L(t)v_1 \in C([0,T],L^2(\Omega))\,.
\end{equation}
To establish the optimal regularity of the right hand side in \eqref{e:volterra-v_tt},
we pinpoint the one of $H_{tt}$.
We resume then \eqref{e:derive-affine_1bis}, derive one more time with respect to $t$, thus obtaining
\begin{equation*}
\begin{split}
H_{tt}(t)&=\cA^2 R_+(t)\big[w_0-Dg(0)\big]
+\cA R_-(t)\Big[\frac{\gamma}{2}\big(w_0-Dg(0)\big)+\big(w_1-Dg_t(0)\big)\Big] 
\\[1mm]
& \quad 
-\cA\int_0^t R_-(t-s) D\tilde{g}_{tt}(s)\,ds + R_+(t)\cF(0) 
\\[1mm]
& \quad 
+\int_0^t R_+(t-s)\big[\cF'(s)+ \tilde{f}'(s)\big]\,ds\,,
\end{split}
\end{equation*}
where all summands readily belong to $C([0,T],L^2(\Omega))$.
(Again, the term $D\tilde{g}_{tt}(t)$ and its opposite produce a cancellation.)
%
The obtained regularity $H_{tt}\in C([0,T],L^2(\Omega))$ combined with 
\eqref{e:trivial} implies the same membership for the right hand side
of \eqref{e:volterra-v_tt} and hence for $v_{tt}$ and $w_{tt}$.
\qed

\smallskip
\noindent
\paragraph{\em Boundary regularity results.}
To pinpoint the regularity of the normal traces, assume the stronger hypothesis $g\in H^2(\Sigma)$. We resume the IBVP \eqref{e:ibvp-for-v} satisfied by $v=e^{\gamma t/2}w$,
where $\cF(t)$ denotes the function in \eqref{e:calf} that depends on initial data
$(w_0,w_1,w_2)$; 
the present initial and boundary data $(v_0,v_1,\tilde{g})$ are defined in terms of $(w_0,w_1,g)$ as in \eqref{e:data-for-v}.
Thus, denote by $\hat{f}(\cdot)$ the sum
\begin{equation*}
\hat{f}(t) = \underbrace{\beta v(t)+\displaystyle\int_0^t K(t-s)v(s) ds}_{f_0(t)}
+ \underbrace{h_0(t)w_0+h_1(t)w_1}_{f_1(t)}
+ \underbrace{h_2(t)\big(w_2-\Delta w_0\big)}_{f_2(t)}
+ \tilde{f}(t)\,,
\end{equation*}
that we will consider as an affine term in \eqref{e:ibvp-for-v}.
We then have $v(t) = z(t) + v_2(t)$, having set
\begin{equation}\label{e:satisfies-wave}
\begin{split}
z(t)&:=R_+(t)v_0+\cA^{-1}R_-(t)v_1+\cA^{-1}\int_0^t  R_-(t-s)\big[f_0(s)+f_1(s)\big]\,ds
\\
& \myspace 
- \cA\int_0^t R_-(t-s) D\tilde{g}(s)\,ds\,,
\\[1mm]
v_2(t)&:=\cA^{-1}\int_0^t R_-(t-s) \big[f_2(s)+\tilde{f}(s)\big]\,ds
=: \big[\cK\big(f_2(\cdot)+\tilde{f}(\cdot)\big)\big](t)\,.
\end{split}
\end{equation}
The analysis of the (sharp) boundary regularity of the summand $z$ is straightforward:
indeed, in view of the regularity assumed on $(w_0,w_1,g)$ we know that 
$v_0\in H^2(\Omega)$, $v_1\in H^1(\Omega)$, $\tilde{g}\in H^2(\Sigma)$ (with all meaningful compatibility conditions), while it is
easily seen that $f_0+f_1$ belongs to $C^\infty([0,T],H^1(\Omega))\subset L^2(0,T;H^1(\Omega))$.
(To ascertain the latter claim, observe that $f_0(\cdot)$ possesses the interior regularity
of $v(\cdot)$, i.e.~$v\in C([0,T],H^2(\Omega))$ -- in view of part {\em (a)} of the proof --, which is even stronger than required,
whilst $f_1\in C^\infty([0,T],H^1(\Omega))$ because 
$w_1, w_0\in H^1(\Omega)$ and $h_i\in C^\infty$, $i=0,1$.)
Consequently, the trace result 
\begin{equation}\label{e:stronger-trace}
\frac{\partial z}{\partial \nu}\in H^1(\Sigma)= L^2(0,T;H^1(\Omega))\cap H^1(0,T;L^2(\Gamma))
\end{equation}
is valid (\cite{sakamoto82}, \cite[Theorem 2.2]{las-lions-trig_1986}); in particular, 
\begin{equation}\label{e:boundary-reg_1}
\frac{\partial^2 z}{\partial t\,\partial \nu}\in L^2(\Sigma)\,.
\end{equation}
Let us prove that $\frac{\partial v_2}{\partial \nu}\in H^1(0,T;L^2(\Gamma))$,
as well.
Rewrite $v_2=v_{21}+v_{22}$, and examine the first summand
\begin{equation*}
v_{21}(t):=\big(\cK f_2\big)(t)
= \cA^{-1}\int_0^t h_2(t-s) R_-(s)(w_2-\Delta w_0)\,ds\,.
\end{equation*}
Since $h_2\in C^1$, then
\begin{equation*}
\begin{split}
v_{21t}(t) &= \frac{\partial}{\partial t}v_{21}(t)
= \cA^{-1}h_2(0)R_-(t)(w_2-\Delta w_0) 
+\cA^{-1}\int_0^t h_2'(t-s)R_-(s) (w_2-\Delta w_0)\,ds
\\[1mm]
& = \cA^{-1}R_-(t)h_2(0)(w_2-\Delta w_0)
+\cA^{-1}\int_0^t R_-(t-s) h_2'(s)(w_2-\Delta w_0)\,ds\,,
\end{split}
\end{equation*}
which shows that $v_{21t}(t)$ is the solution to an IBVP for a second-order linear wave equation such as \eqref{e:ibvp-wave}, with specifically 
\begin{equation*}
\begin{split}
z_0=0\,, \quad z_1=h_2(0)\big(w_2-\Delta w_0\big)\in L^2(\Omega)\,, 
\\[1mm]
f(t) = h_2'(s)\big(w_2-\Delta w_0\big)\in C^\infty([0,T],L^2(\Omega))\,, \quad g\equiv 0\,.
\end{split}
\end{equation*}
The regularity of initial and boundary data, along with the compatibility condition 
$z_0|_\Gamma= g|_{t=0}=0$ ensure then
\begin{equation*}
\frac{\partial v_{21t}}{\partial \nu}\in L^2(\Sigma)
\end{equation*}
({\sl cf.}~\cite{sakamoto82}, \cite{las-lions-trig_1986}).
A similar computation performed for the second component 
$v_{22}(t)=\big(\cK \tilde{f}\big)(t)$ -- where it is utilized that
$\tilde{f}\in H^1(0,T;L^2(\Omega))$ -- allows to confirm 
the same regularity of its normal trace on $\Gamma$. 
Consequently,
\begin{equation}\label{e:boundary-reg_2}
\frac{\partial^2 v_2}{\partial t\,\partial \nu}\in L^2(\Sigma)
\end{equation}
\smallskip
as well.

Recalling that $w=e^{-\gamma t/2}v$ and $v=z+v_2$, in view of the boundary regularity results \eqref{e:boundary-reg_1} and \eqref{e:boundary-reg_2} established for either summand, the very same regularity is valid for $w$.
The trace regularity result is stated explicitly in the following Proposition.

\begin{proposition}\label{p:trace-partial}
Under the hypotheses of Theorem~\ref{t:main}, part (a) with the stronger assumption $g\in H^2(\Sigma)$ on the Dirichlet data, the following boundary regularity result holds true:
\begin{equation*}
\frac{\partial^2 w}{\partial t\,\partial \nu}\in L^2(\Sigma)\,.
\end{equation*}

\end{proposition}

\bigskip
Return now to $v=z+v_2$ and recall that the first summand $z$ satisfies the full trace result \eqref{e:stronger-trace} anyhow.
We thus examine the summand $v_2$ and observe that if $f\in L^1(0,T;H^1(\Omega))$,
we also have $\tilde{f}\in L^1(0,T;H^1(\Omega))$ (the time regularity is better, actually) so that
\begin{equation*}  
\frac{\partial}{\partial \nu}\big(\cK \tilde{f}\big)\in H^1(\Sigma)
\end{equation*}
(\cite{sakamoto82}, \cite[Theorem 2.2]{las-lions-trig_1986}).
As for the component $v_{21}$, the same arguments yield
\begin{equation*}  
\frac{\partial v_{21}}{\partial \nu}
=\frac{\partial}{\partial \nu}\big[\cK \big(h_2(\cdot)(w_2-\Delta w_0\big)\big]
\in H^1(\Sigma)\,,
\end{equation*}
provided $w_2-\Delta w_0\in H^1(\Omega)$.

\begin{remark}\label{r:trace-partial}
\begin{rm}
One the basis of these last observations, it appears that the approach taken in the present section enables us to achieve the full trace result $\partial u/\partial \nu \in H^1(\Sigma)$ under the hypotheses of Theorem~\ref{t:main}, part (a), provided that
\begin{equation*}
g\in H^2(\Sigma)\,, \qquad
f\in L^1(0,T;H^1(\Omega))\,, \qquad w_2-b\Delta w_0\in H^1(\Omega)\,,
\end{equation*}
which is true e.g. in the case $w_0=w_2=0$, still with $f\in L^1(0,T;H^1(\Omega))$.
\end{rm}
\end{remark}


\section{An approach based on the theory of hyperbolic equations: 
Proof of Theorem~\ref{t:main}, part (b)} 
\label{s:hyperbolic-systems}
Even though the characteristic boundary will force us to make some adjustments to the general theory of mixed problems for hyperbolic equations, the overall procedure will be the same as described by Sakamoto \cite[Chapter 3]{sakamoto82}. We start by deriving a {\it resolvent estimate} for the operator $M$. A duality argument establishes then the existence and uniqueness of the initial-boundary value problem with homogeneous initial data. In order to include inhomogeneous initial data one needs to extend the resolvent estimate to a {\it semigroup estimate}. 

Throughout this section we will work with the infinite time interval and denote $Q_\infty = \R \times \Omega$ and $\Sigma_\infty=\R \times \Gamma$. 
From the theory of hyperbolic equations, we know the estimate
\[
 \beta \|e^{-\beta t}v\|^2_{1,\beta,Q_\infty} + \|e^{-\beta t}\partial_\nu v\|^2_{\Sigma_\infty} \lesssim \frac{1}{\beta} \|e^{-\beta t}(\partial_t^2 - b\Delta)v\|_{Q_\infty}^2 +\|e^{-\beta t}v\|^2_{1,\beta,\Sigma_\infty}
\]
for sufficiently large $\beta$ and all $v$ satisfying $e^{-\beta t}v \in H^2(Q)$, see e.g. \cite[formula (24.1.4)]{hormander85} Here the norms are weighted Sobolev norms
\[
 \|u\|^2_{k,\beta,Q_\infty} = \sum_{|\alpha|\le k} \beta^{2k-2|\alpha|}\|\partial^\alpha u\|^2_{L^2(Q_\infty)}\quad  \mbox{ and } \quad \|u\|^2_{k,\beta,\Sigma_\infty} = \sum_{|\alpha|\le k} \beta^{2k-2|\alpha|}\|\partial^\alpha u\|^2_{L^2(\Sigma_\infty)}\;,
\]
where the derivatives in the second norm are all tangential derivatives. 

 Choose $w \in H^3(Q)$ and set $v=\partial_t w$. The estimate above appear then as  
\[
 \beta \|e^{-\beta t}w_t\|^2_{1,\beta,Q_\infty} + \|\partial_\nu w_t\|^2_{\Sigma_\infty} \lesssim \frac{1}{\beta} \|e^{-\beta t}(\partial_t^2 - b\Delta)w_t\|_{Q_\infty}^2 +\|e^{-\beta t}w_t\|^2_{1,\beta,\Sigma_\infty}\;,
\]
which is already an estimate for the principal part of the MGT operator. However, we estimate only time-derivatives. Hence, we add to this estimate the original estimate of the wave operator with $v=w$, multiplied with $\beta^2$
and obtain
\begin{multline*}
 \beta^3\|e^{-\beta t}w\|^2_{1,\beta,Q_{\infty}}+\beta \|e^{-\beta t}w_t\|^2_{1,\beta,Q_\infty} + \beta^2\|e^{-\beta t}\partial_\nu w\|^2_{\Sigma_\infty}+ \|e^{-\beta t}\partial_\nu w_t\|^2_{\Sigma_\infty} \\ \lesssim \frac{1}{\beta} \|e^{-\beta t}(\partial_t^2- b\Delta)w_t\|^2_{Q_\infty} + \beta\|e^{-\beta t}(\partial_t^2 - \Delta)w\|_{Q_\infty} + \|e^{-\beta t}w_t\|^2_{1,\beta, \Sigma_\infty} + \beta^2 \|e^{-\beta t}w\|^2_{1,\beta,\Sigma_\infty}\;.
\end{multline*}
Using now the fact that $\partial_t$ is a hyperbolic operator we estimate
\begin{equation}\label{2}
 \beta \| e^{-\beta t}(\partial_t^2 - b\Delta)w\|^2_{Q_\infty} \lesssim \frac{1}{\beta} \|e^{-\beta t}(\partial_t^2 - b\Delta)w_t\|_{Q_\infty}^2
\end{equation}
which improves the estimate above to  
\begin{multline}\label{3}
 \beta \| e^{-\beta t} (\beta w,w_t)\|_{1,\beta,Q_\infty}^2 + \|e^{-\beta t} (\beta \partial_\nu w, \partial_\nu w_t)\|^2_{\Sigma_\infty}\\ \lesssim \frac{1}{\beta}\|e^{-\beta t}(\partial_t^3 - b\Delta\partial_t)w\|_{Q_\infty}^2 +\|e^{-\beta t}(\beta w,w_t)\|^2_{1,\beta,\Sigma_\infty}\;.
\end{multline}
Next we will show that in the last estimate we can replace the operator of third order $\partial_t^3 - b\Delta\partial_t$ which is the principal part of the MGT equation, by the operator $M$. This is not a triviality here since our estimate does not have all derivatives of second-order on the left hand side. From (\ref{2}) and the triangle inequality we infer that 
\[
 \beta \| e^{-\beta t}\Delta w\|^2_{Q_\infty}\lesssim \beta \| e^{-\beta t}(\partial_t^2 - b\Delta)w\|^2_{Q_\infty} + \beta \| e^{-\beta t}\partial_t^2 w\|^2_{Q_\infty}\;, 
\]
which tells us that we can include the Laplacian of $w$ on the left-hand side in (\ref{3}), that is 
\begin{multline*}
 \beta \| e^{-\beta t}\Delta w\|^2_{Q_\infty}+
 \beta \| e^{-\beta t} (\beta w,w_t)\|_{1,\beta,Q_\infty}^2 + \|e^{-\beta t} (\beta \partial_\nu w, \partial_\nu w_t)\|^2_{\Sigma_\infty}\\ \lesssim \frac{1}{\beta}\|e^{-\beta t}(\partial_t^3 - b\Delta\partial_t)w\|_{Q_\infty}^2 +\|e^{-\beta t}(\beta w,w_t)\|^2_{1,\beta,\Sigma_\infty}\;.
\end{multline*}
Since the lower order term of the MGT equation is combination of the Laplacian and a second-order time derivative, this last estimate -- in connection with the triangle inequality -- provides our basic estimate of the MGT operator $M$, that is, 
\begin{equation}\label{1i}
  \beta \| e^{-\beta t} (\beta w,w_t)\|_{1,\beta,Q_\infty}^2 + \|e^{-\beta t} (\beta \partial_\nu w, \partial_\nu w_t)\|^2_{\Sigma_\infty} \lesssim \frac{1}{\beta}\|e^{-\beta t}Mw\|_{Q_\infty}^2 +\|e^{-\beta t}(\beta w,w_t)\|^2_{1,\beta,\Sigma_\infty}\;.
\end{equation} 
It may be worthwhile to compare the estimate (\ref{1i}) with the standard resolvent estimates for hyperbolic boundary problems in the non-characteristic case \cite[Section 3.3]{sakamoto82}. While we do not manage to estimate neither the $H^2(Q_\infty)$ norm of $w$ nor the $H^1(\Sigma_\infty)$ of the exterior normal derivative on the boundary, our estimate does not require the $H^2(\Sigma_\infty)$ norm $w$ on the  right-hand side. Furthermore, the $L^2$-norm of $\partial^2_\nu w$ cannot occur in the left-hand side, since the operator $\partial_t^3-\Delta\partial_t$ does not involve any normal derivative of order 3. What is more surprising is that the $H^1$-norm of the normal derivative (of order one) does not occur on the left-hand side in (\ref{1i}). On the other hand, even the boundary term on the right-hand side is different than in the non-characteristic case.

We also wish to point out that we can obtain estimate (\ref{1i}) by using  micro-local analysis, as in \cite[Section 3.3]{sakamoto82}. The Dirichlet boundary operator satisfies the Kreiss-Sakamoto condition (uniform Lopatinskii condition) with respect to the operator $M$. However, we feel that our approach is more direct, also since the initial estimate for the wave operator can be established by energy integrals \cite[Section 24.1]{hormander85}.

Furthermore, estimates similar to (\ref{1i}) can be derived. For example, one can use elliptic estimates to include second-order space derivatives which will force us to include the norm in $L^2(\R,H^{3/2}(\Gamma))$ of $w$ in the right-hand side. This way we could provide an alternative proof to Theorem \ref{t:main}, part~(a) and we would gain complete flexibility with respect to lower-order terms. While the notation of a finite energy solution is clearly defined in the non-characteristic case, it seems that in the characteristic case several avenues can be pursued. 

 At this point in time we do not know whether the more complete estimate (with respect to the traces)  
\[
 \|e^{-\beta t} w\|_{2,\beta,Q_\infty}^2 + \|e^{-\beta t} \partial_\nu w\|^2_{1,\beta,\Sigma_\infty} \lesssim \frac{1}{\beta}\|e^{-\beta t}Mw\|_{Q_\infty}^2 +\|e^{-\beta t}w\|^2_{2,\beta,\Sigma_\infty}  
\] 
holds.

Now we will discuss, how the apriori estimate (\ref{1i}) can be used to obtain and existence-and-uniqueness statement for our initial-boundary-value problem. Note at first that, by setting $u=e^{-\beta t}w$, we have 
\begin{equation}\label{2}
 \beta\|(\beta u,u_t)\|_{1,\beta,Q_\infty}^2 +
 \|(\beta \partial_\nu u,\partial_\nu u_t)\|^2_{\Sigma_\infty} \lesssim \frac{1}{\beta}\|M_\beta u\|_{Q_\infty}^2 +
 \|(\beta u,u_t)\|^2_{1,\beta,\Sigma_\infty}\;,
\end{equation} 
where 
\[
 e^{-\beta t} M w = (\partial_t+\beta)^3 u - b\Delta (\partial_t+\beta) u +\alpha (\partial_t +\beta)^2u -c^2\Delta u = M_\beta u\;.
\] 
Since we will need to scale our basic estimate (\ref{1i}) to different tangential Sobolev levels, it will be convenient to introduce semigeodesic local coordinates. Localizing the function, we may assume that $\Omega_\infty = \{ x_d >0\}$ and $\Sigma_\infty = \{x_d = 0\}$. The spatially tangential variables are denoted by $y=(x_1,...,x_{d-1})$ and we define a family of tangential operators 
\[
 \Lambda^s w = \frac{1}{(2\pi)^d} \int_{\R^d} e^{{\rm i}[\tau t + \langle y,\eta\rangle]} (\tau^2+\beta^2+|\eta|^2)^{s/2}\hat{w}(\tau,\eta,x_d) \,d\tau d\eta \;, \quad \mbox{ for } s\in \R,
\]
where $\hat{w}$ denotes the Fourier transform in the tangential variables $(t,y)$. 
 This estimate can be scaled to other Sobolev norms with respect to the tangential variables. Replacing $u$ by $\Lambda^s u$, we have 
\begin{equation}\label{4}
 \beta\| \Lambda^s (\beta u,u_t)\|^2_{1,\beta,Q_\infty} +
 \|\Lambda^s (\beta\partial_d u,\partial_d u_t)\|^2_{1,\beta,\Sigma_\infty} \lesssim \frac{1}{\beta} \|\Lambda^s M_\beta u\|^2_{Q_\infty} + \|\Lambda^s (\beta u,u_t)\|^2_{1,\beta,\Sigma_\infty}\;.
\end{equation}
for any $s\in \R$ and $u\in H_{(3,s)}(Q_\infty)$ \cite[Lemma 3.3.7]{sakamoto82} and sufficiently large $\beta$.  The anisotropic Sobolev space $H_{(m,s)}(Q_\infty)$  has been introduced in \cite[Appendix B2]{hormander85}.

In semi-geodesic coordinates, the Laplacian will transfer to the Riemann Laplacian
\[
 \Delta_a = \frac{1}{\sqrt{\det a(x)}}\sum_{j=1}^{d}\partial_j\sqrt{\det a(x)} a^{jk}(x)\partial_{k}\;,
\]
where $a^{jd}=\delta_{jd}$, $a^{jk}=a^{kj}$ for $j,k=1,...,d$ are smooth functions. Hence, the changes in the operator $M$ are such that the Laplacian $\Delta$ is replaced by the Riemann Laplacian $\Delta_a$. The estimates (\ref{2}) and (\ref{4}) are then established at first locally and then combined by means of a partition of unity. 

For future reference we point out that the Riemann Laplacian can be written as
\begin{equation}\label{8}
 \Delta_a = \frac{1}{\sqrt{\det a}}\partial_d \sqrt{\det a} \partial_d+ \Delta'_a
\end{equation}
where $\Delta'_a$ is the Laplace-Beltrami operator on the surfaces $\{x_d = c\}$ for small positive $c$.

From the estimate (\ref{4}) we can derive an existence statement for the boundary problem
\begin{equation}\label{5}
 M_\beta u = e^{-\beta t}f  \mbox{ in } Q_\infty,\quad  u\big|_{Q_\infty} =e^{-\beta t}g  \mbox{ in } \Sigma_\infty
\end{equation}
by means of duality.
For $u,z\in H^3(Q_\infty)$ one obtains using integration by parts the identity
\[
 ( M_\beta u,z )_{Q_\infty} = (u,M^*_\beta z) _{Q_\infty} + \langle u, c^2 \partial_\nu z-b\partial_\nu z_t \rangle_{\Sigma_\infty} +\langle c^2\partial_\nu u +b\partial_\nu w_t, z\rangle_{\Sigma_\infty} \;.
\]
Here $(\cdot,\cdot)_{Q_\infty}$ is the scalar product in $L^2(Q_\infty)$ and $\langle\cdot,\cdot\rangle_{\Sigma_\infty}$ denotes the scalar product in $L^2(\Sigma_\infty)$.

The resolvent estimate for the adjoint operator is 
\[
 \beta\| (\beta z,z_t)\|^2_{1,\beta,Q_\infty} + 
 \|(\beta \partial_\nu z,\partial_\nu z_t)\|^2_{\Sigma_\infty} \\ \lesssim \frac{1}{\beta} \|M^*_\beta z\|^2_{Q_\infty} +
 \|(\beta z,z_t)\|^2_{2,\beta,\Sigma_\infty}\;.
\]
for all $z\in H^3(Q_\infty)$ and $\beta \ge \beta_1>0$. This estimate is established in a similar fashion as (\ref{1i}) since the two operators $M^*_\beta$ and $-M_{-\beta}$ have the same principal part. As the estimate for the primal problem, this one can be also scaled to other Sobolev norms in the tangential variables. For $s\in \R$ and $z\in H_{(3,s)}(Q_\infty)$, we have, for $\beta$ sufficiently large, 
\begin{equation}\label{6}
 \beta\| \Lambda^s(\beta z,z_t)\|^2_{1,\beta,Q_\infty}+ 
 \|\Lambda^s(\beta \partial_\nu z,\partial_\nu z_t)\|^2_{\Sigma_\infty} \lesssim \frac{1}{\beta} \|\Lambda^s M^*_\beta z\|^2_{Q_\infty} 
 + \|\Lambda^s(\beta  z,z_t)\|^2_{1,\beta,\Sigma_\infty}\;.
\end{equation}
This estimate is crucial when proving the existence of solutions to the boundary problem (\ref{5}), \cite[Theorem 3.4(b)]{sakamoto82}.
 

\begin{proposition}\label{p:matthias-best}
Let $s$ be a non-negative integer and suppose that $\beta$ is sufficiently large that the two estimates (\ref{4}) and (\ref{6}) hold. For $e^{-\beta t} f\in H^s(Q_\infty)$ and $e^{-\beta t} (g,g_t) \in H^{s+1}(\Sigma_\infty)$, the boundary value problem (\ref{5}) has a unique weak solution $u\in H^{1+s}(Q_\infty)$.
This solution is the strong limit of functions which are smooth. It has the additional regularity properties $u_t \in H^{s+1}(Q_\infty)$ and $\partial_\nu  u,\partial_\nu u_t\in H^{s}(\Sigma_\infty)$, and the estimate
\[
 \beta \|(\beta u,u_t)\|^{2}_{1+s,\beta,Q_\infty} + \|(\beta\partial_\nu u,\partial_\nu u_t)\|^2_{s,\beta,\Sigma_\infty} \lesssim
  \frac{1}{\beta}\|e^{-\beta t}f\|^2_{s,\beta,Q_\infty} + \|e^{-\beta t}(\beta g,g_t)\|^2_{1+s,\beta,\Sigma_\infty}
\]
holds.
\end{proposition}
\begin{proof}
 In the first step we prove the existence of a weak solution using a duality argument. For that we define
 \[
  Y = \left\{ z\in H_{(3,-s-1)}(Q_\infty)\;:\; z\Big|_{\Sigma_\infty}=0\right\} 
 \]
 and $Z = \{ M_\beta^*z\,:\, z\in Y\}$. Note that $Z$ is a subspace of $H_{(0,-s-1)}(Q_\infty)$. Consider the linear functional  defined by 
 \[
  l(z) = (z,e^{-\beta t}f)_{Q_\infty} - c^2 \langle \partial_d z,e^{-\beta t}g \rangle_{\Sigma_\infty} + b \langle \partial_d z_t, e^{-\beta t}g \rangle_{\Sigma_\infty}\;.
 \]
 This linear functional is bounded on $Z$ since 
 \[
  \begin{split}
  |l(z)|&\lesssim \|e^{-\beta t}f\|_{s,\beta,Q_\infty} \|\Lambda^{-s}z\|_{Q_\infty} + \|e^{-\beta t} g\|_{s+1,\beta,\Sigma_\infty}
   \|\left[\|\Lambda^{-s-1}\partial_d z_t\|+\|\Lambda^{-s-1}\partial_d z\|\right]\\
   &\lesssim \left[ \|e^{-\beta t}f\|^2_{s,\beta,Q_\infty}+\|e^{-\beta t} g\|_{s+1,\beta,\Sigma_\infty}\right] \| \Lambda^{-s-1}M_\beta z\|\;,
  \end{split}
 \]  
 where we applied formula (\ref{6}) with $s$ replaced by $-s-1$. Applying the Theorems by Hahn-Banach and by Riesz, there exists an element $u\in H_{(0,s+1)}(Q_\infty)$ such that  $l(z) = (M_\beta z,u)_{Q_\infty}$, that is 
\[
   (M_\beta z,u)_{Q_\infty}=(z,e^{-\beta t}f)_{Q_\infty} - c^2 \langle \partial_d z,e^{-\beta t}g \rangle_{\Sigma_\infty} + b \langle \partial_d z_t, e^{-\beta t}g \rangle_{\Sigma_\infty}
\]
for all $z\in Y$. Hence, $u$ is a weak solution to the boundary value problem. 
 
In the second step, the regularity in the direction normal to the boundary is established. From equation $M_\beta u =e^{-\beta t}f$ and the Laplacian in local coordinates (\ref{8}) we know that 
\begin{multline}\label{7}
\frac{c^2+\beta b}{\sqrt{\det{a}}}\partial_d(\sqrt{\det a}\partial_d u) + \frac{b}{\sqrt{\det a}} \partial_d(\sqrt{\det a}\partial_{d} u_t) \\ =(\partial_t +\beta)^3 u+\alpha (\partial_t +\beta)^2u - (c^2-b \beta )\Delta'_a u - b\Delta'_a u_t - e^{-\beta t} f\;.
 \end{multline}
This equation can be considered as a ordinary differential equation for the unknown 
$\partial_d(\sqrt{\det g}\partial_d )u$ with respect to $t$. In the following we abbreviate the right-hand side in this equation by $h$ and note that $h \in H_{(0,s-2)}(\R^{d+1}_+)$. Integrating in time gives 
 \[
   \frac{1}{\sqrt{\det a}}\partial_d(\sqrt{\det a}\partial_d u)(t,x) =\frac{1}{b} \int_{-\infty}^t  e^{-(c^2/b+\beta)(t-s)} h(s,x)\,ds 
 \] 
 and hence  $\partial_d(\sqrt{\det a}\partial_d u) \in H_{(0,s-1)}(\R^{d+1}_+)$ since $h$ has time derivatives of order three but space derivatives of order two. Now note that the identity
 \[
  \frac{1}{(\det a)^{1/4}}\partial_d^2 \left[(\det a)^{1/4} u\right] = \frac{1}{\sqrt{\det a}}\partial_d\left[\sqrt{\det a}\partial_d u\right] + u \frac{\partial_d^2 (\det a)^{1/4}}{(\det a)^{1/4}}
 \]  
 implies $(\det a)^{1/4} u \in H_{(2,s-1)}(\R^{d+1}_+)$ and thus $u\in H_{(2,s-1)}(\R^{d+1}_+)$, because of the smoothness of $\det a$. 
 
 This process can be repeated and results in $u \in H^{s+1}(Q_\infty)$. Furthermore, derivatives in normal direction can be estimated in terms of tangential derivatives and the forcing term. These facts can be combined into an improvement of the estimate (\ref{4}). For $u \in H^{s+3}(Q_\infty)$ we have 
\begin{equation}\label{9}
 \beta \|(\beta u,u_t)\|^{2}_{1+s,\beta,Q_\infty} + \|(\beta \partial_\nu u,\partial_\nu u_t\|^2_{1+s,\beta,\Sigma\infty} \lesssim
  \frac{1}{\beta}\|M_\beta u\|^2_{s,\beta,Q_\infty} + \|(\beta u,u_t)\|^2_{1+s,\beta,\Sigma_\infty}
\end{equation}
Finally, we prove that the solution $u \in H^{s+1}(Q_\infty)$ satisfies $u_t\in H^{s+1}(Q_\infty)$,  $\partial_\nu  u,\partial_\nu u_t\in H^{s}(\Sigma_\infty)$, and the estimate stated in the proposition. Approximating the forcing term and the boundary data by functions in $H^{s+1}(Q_\infty)$ and $H^{3+s}(\Sigma)$ we obtain a sequence of solutions in $u^{(n)} \subset H^{2+s}(Q_\infty)$. Using estimate (\ref{9}) on $u^{(n)} - u^{(m)}$ shows that this sequence $u^{(n)}$ and $u_t^{(n)}$ are Cauchy in $H^{s+1}(Q_\infty)$, the sequence  is Cauchy  and that the sequences $\partial_\nu u^{(n)}$ and $\partial_\nu u_t^{(n)}$ are Cauchy in $H^{s}(\Sigma)$. The limit function $u\in H^{1+s}(Q_\infty)$ is then the strong solution and satisfies the estimate.
\end{proof}
Returning to the original variable $w$ we have solved the boundary problem
$Mw=f$ in $Q_\infty$, $u=g$ in $\Sigma_\infty$ whose solution satisfies the estimate
\[
 \beta \|e^{-\beta t}(\beta w,w_t)\|^2_{1+s,\beta,Q_\infty} + \|e^{-\beta t}(\beta \partial_\nu w,\partial_\nu w_t)\|^2_{\Sigma_\infty} \lesssim
  \frac{1}{\beta}\|e^{-\beta t}f\|^2_{s,Q_\infty} + \|e^{-\beta t}(\beta g,g_t)\|^2_{1+s,\beta,\Sigma_\infty\;.}
\]
This estimate can be used to solve the initial-boundary value problem with homogeneous initial data. If $f$ and $g$ both vanish for $t<0$, then $w$ has to vanish for $t<0$ as well, see \cite[Lemma 3.4.4]{sakamoto82}.


Now we discuss non-homogeneous initial data. For $s$, a non-negative integer, and $w \in H^{3+s}(Q)$
one has the {\it semigroup estimate} 
\begin{multline}\label{10}
 \|e^{-\beta T}w(T)\|^2_{1+s,\Omega} + \|e^{-\beta T}w_t(T)\|^2_{1+s,\Omega} + \|e^{-\beta T}w_{tt}(T)\|^2_{s,\Omega} \\+
 \|e^{-\beta t} (w,w_t)\|^2_{1+s,\beta,Q} + \|e^{-\beta t}(\partial_\nu w,\partial_\nu w_t)\|^2_{s,\Sigma}\\  \lesssim \|e^{-\beta t}Mw\|^2_{s,Q} + \|e^{-\beta t}(w,w_t)\|^2_{1+s,\Sigma}+ \|w(0)\|^2_{1+s,\Omega} + \|w_t(0)\|^2_{1+s,\Omega} + \|w_{tt}(0)\|^2_{s,\Omega}\;.
\end{multline}
This semigroup estimate is proved as in Sakamoto's book \cite[Section 3.5]{sakamoto82}. The only difference is that our starting point is not the energy integral 
\[
 \Re \int_Q e^{-2\beta t}Mw(2\partial_t^2-b\Delta) \overline{w} \,dtdx \quad \mbox{ but rather } \quad \Re \int_Q e^{-2\beta t}Mw(2\partial_t^2+\beta^2 b) \overline{w}\,dtdx\;.
\]
This is due to the fact that our resolvent estimate does not have second-order space derivatives on the left-hand side. 
The semigroup estimate plays a crucial role in the proof of the following result which includes part (b) of Theorem \ref{t:main} in the case $s=0$. In order to talk about solutions to the initial-boundary value problem (\ref{e:ibvp-mgt_1}) of higher regularity, we need to discuss compatibility conditions. 

If $f\in H^s(Q)$, $s\ge 1$ then the differential equation $Mw=f$ can be used to recover the initial values of higher-order time derivatives $\partial_t^{l+2} w(0)$, $l=1,2,...,s$. If $\partial_t^{l} g(0) = \partial_t^{l} w(0)$ on $\Gamma$ for $l=0,1,...,s+2$, then the data $f$, $g$, $w^0$, $w^1$, $w^2$ satisfy the compatibility condition of order $s$.
The compatibility conditions of order $s=0$ are the ones stated in Theorem \ref{t:main}. 
\begin{proposition}
 Let $s$ be a non-negative integer.
 Consider the initial-boundary value problem (\ref{e:ibvp-mgt_1}) with $f\in H^s(Q)$, $g,g_t\in H^{s+1}(\Sigma)$, and $w_0, w_1\in H^{s+1}(\Omega)$, $w_2 \in H^s(\Omega)$ satisfying the compatibility conditions of order $s$. There exists a unique solution $(w,w_t)\in C([0,T], H^{s+1}(\Omega)\times H^{s+1}(\Omega))$ with additional trace regularity $\partial_\nu w, \partial_\nu w_t \in H^{s}(\Sigma)$ and the estimate
\begin{multline*}
 \|(w(T),w_t(T))\|^2_{1+s,\Omega} + \|w_{tt}(T)\|^2_{s,\Omega} +
 \|(w,w_t)\|^2_{1+s,Q} + \|(\partial_\nu w, \partial_\nu w_t)\|^2_{1+s,\Sigma} \\ \lesssim  \|f\|^2_{s,Q} + \|(g,g_t)\|^2_{1+s,\Sigma}+ \|(w_0,w_1)\|^2_{1+s,\Omega} + \|w_2\|^2_{s,\Omega}\;,
\end{multline*}
holds.  
\end{proposition} 
\begin{proof}
 The forcing term $f$, the Dirichlet data $g$ and the initial data $w_0,w_1,w_2$ can be all approximated by sequences of functions $f^{(n)} \in H^{s+2}(Q)$, $g^{(n)}\in H^{s+4}(\Sigma)$, $w^{(n)}_0 \in H^{s+9/2}$, $w^{(n)}_1 \in H^{s+7/2}(\Omega)$, $w_2^{(n)}\in H^{s+5/2}(\Omega)$ and satisfy the compatibility conditions of order $s+2$ \cite[Lemma 3.3]{rm74}. By the trace theorem in Sobolev spaces, there exists a sequence $w^{(n)}_{II}\in H^{s+5}(Q)$ such that 
\[
  \partial_t^j w^{(n)}_{II} (0) = w_j^{(n)}(0)\quad \mbox{ for } j=0,1,2.
\]
Let $w^{(n)}_I$ be the solution to the initial-boudary value problem
\[
 Mw=f^{(n)}-Mw^{(n)}_{II}  \mbox{ in } Q\,, \quad w= g^{(n)}-w^{(n)}_{II} \mbox{ in } \Sigma\,, \quad w(0)=w_t(0)=w_{tt}(0)=0\;.
\]
According to Proposition \ref{p:matthias-best} and the discussion right after its proof, this problem has a unique solution $w_I^{(n)} \in H^{s+3}(Q)$. Consequently, the function $w^{(n)}:=w^{(n)}_I + w^{(n)}_{II}$ satisfies the initial-boundary value problem
\[
 Mw=f^{(n)} \mbox{ in } Q, \quad w=g^{(n)} \mbox{ in } \Sigma, \quad w(0)=w_0^{(n)}, w_t(0)=w_1^{(n)}, w_{tt}(0)=w_2^{(n)} \mbox{ in } \Omega
\] 
Using the semigroup estimate (\ref{10}), the difference $w^{(n)}-w^{(m)}$ is shown to be Cauchy in $H^{s+1}(Q)$. Its limit $w\in H^{s+1}(Q)$ is then the solution to the  initial-boundary value problem (\ref{e:ibvp-mgt_1}) and the estimate and the additional regularity statements follow from (\ref{10}) applied to $w^{(n)}$ and taking the limit.
\end{proof}
\begin{remark}
\begin{rm}
 Once we consider a finite interval the exponential term $e^{-\beta t}$ as well as the Sobolev norms with the parameter $\beta$ become unnecessary. Hence, we decided to formulate the estimate in the last Proposition without them. In the infinite time interval is considered, one has to use weighted norms and the exponential term. 
\end{rm}
\end{remark}

\section*{Acknowledgements}
A major portion of
this 
research was conducted while Matthias Eller was the guest of Francesca Bucci at the Universit\`a degli Studi di Firenze (UniFI). 
Eller's stay was supported by grants from the Gruppo Nazionale per l'Analisi Matematica, la Probabilit\`a e le sue Applicazioni (GNAMPA) of the Istituto Nazionale di Alta Matematica, within the programme ``Professori visitatori'', as well as from the ``Internazionalizzazione'' plan of UniFI, which both authors gratefully acknowledge. 
Matthias Eller wishes to thank the Dipartimento di Matematica e Informatica, UniFI,
for its hospitality.
Francesca Bucci is a member of the GNAMPA and participant to the 2019 GNAMPA Project
``Controllability of PDEs in physics models and applied sciences'', whose partial support is also acknowledged. 

\smallskip
Finally, we wish to thank the anonymous referee who made important suggestions and comments which helped us to improve the manuscript.  

\bibliography{article}

\end{document}